\documentclass[a4paper,11pt]{amsart}

\usepackage[margin=1.1in]{geometry}
\usepackage{amsmath,amssymb,amsthm,enumitem,xcolor,hyperref,mathtools}
\usepackage[british]{babel}

\title{Isomorphism classification of Leary--Minasyan groups}
\author{Motiejus Valiunas}
\address{Instytut Matematyczny, Uniwersytet Wroc{\l}awski, plac Grunwaldzki 2/4, 50-384 Wroc{\l}aw, Poland}
\email{valiunas@math.uni.wroc.pl}


\theoremstyle{plain}
\newtheorem{lem}{Lemma}[section]
\newtheorem{prop}[lem]{Proposition}
\newtheorem{thm}[lem]{Theorem}
\newtheorem{cor}[lem]{Corollary}

\counterwithin*{cm}{lem}
\theoremstyle{definition}
\newtheorem{ex}[lem]{Example}
\newtheorem*{ack}{Acknowledgements}
\theoremstyle{remark}
\newtheorem{rmk}[lem]{Remark}

\newcommand{\Q}{\mathbb{Q}}
\newcommand{\Z}{\mathbb{Z}}
\newcommand{\R}{\mathbb{R}}
\newcommand{\C}{\mathbb{C}}
\newcommand{\1}[1]{\overline{#1}}
\DeclareMathOperator\Stab{Stab}
\DeclareMathOperator\Gal{Gal}

\DeclareMathOperator\lcm{lcm}
\DeclareMathOperator\Aut{Aut}

\begin{document}

\begin{abstract}
Recently \cite{leary-minasyan}, I.~J.~Leary and A.~Minasyan studied the class of groups $G(A,L)$ defined as commensurating HNN-extensions of $\Z^n$. This class, containing the class of Baumslag--Solitar groups, also includes other groups with curious properties, such as being CAT(0) but not biautomatic. In this paper, we classify the groups $G(A,L)$ up to isomorphism.
\end{abstract}

\maketitle

\section{Introduction}

Let $n \geq 0$, let $A \in GL_n(\Q)$, and let $L \leq \Z^n \cap A^{-1}(\Z^n)$ be a finite index subgroup. In \cite{leary-minasyan}, I.~J.~Leary and A.~Minasyan defined a group $G(A,L)$ by the presentation
\begin{equation} \label{eq:pres}
G(A,L) = \langle x_1,\ldots,x_n,t \mid [x_i,x_j]=1 \text{ for } 1 \leq i < j \leq n, t\mathbf{x}^{\mathbf{v}}t^{-1} = \mathbf{x}^{A\mathbf{v}} \text{ for } \mathbf{v} \in L \rangle,
\end{equation}
where we write $\mathbf{x}^{\mathbf{w}} := x_1^{w_1} \cdots x_n^{w_n}$ for $\mathbf{w} = (w_1,\ldots,w_n) \in \Z^n$. We refer to the group $G(A,L)$ as a \emph{Leary--Minasyan group}. The class of Leary--Minasyan groups contains the class of Baumslag--Solitar groups: indeed, for $n = 1$, if $L = p\Z$ and $A = \begin{pmatrix} q/p \end{pmatrix} \in GL_1(\Q)$ for some $p,q \in \Z_{\neq 0}$, then $G(A,L) = \langle x,t \mid tx^pt^{-1} = x^q \rangle \cong BS(p,q)$.

Baumslag--Solitar groups often arise as examples or counterexamples: for instance, they were constructed in \cite{bs} to give examples of finitely generated one-relator non-Hopfian groups, and also give finitely presented examples of groups that are Hopfian but not residually finite \cite{bs,meskin,cl} or have infinitely generated automorphism groups \cite{cl}. Likewise, Leary--Minasyan groups were constructed in \cite{leary-minasyan} to give first examples of CAT(0) groups that are not biautomatic. This has prompted the study of both of these classes of groups.

In this paper, we study the conditions under which two given Leary--Minasyan groups are isomorphic. In the case $n = 1$, i.~e.\ for Baumslag--Solitar groups, the isomorphism classification is well known: in particular, given any $p,q,\1p,\1q \in \Z_{\neq 0}$, we have $BS(p,q) \cong BS(\1p,\1q)$ if and only if $\{ \1p,\1q \} = \{ \varepsilon p,\varepsilon q \}$ for some $\varepsilon \in \{ \pm 1 \}$ \cite{moldavanskii}. In \cite[Question~12.3]{leary-minasyan}, I.~J.~Leary and A.~Minasyan asked about classification of the groups $G(A,L)$ up to isomorphism. The following result gives an answer to this problem.

\begin{thm} \label{thm:main}
Let $n,\1n \geq 0$, let $A \in GL_n(\Q)$, $\1A \in GL_{\1n}(\Q)$, and let $L \leq \Z^n \cap A^{-1}(\Z^n)$, $\1L \leq \Z^{\1n} \cap \1A^{-1}(\Z^{\1n})$ be finite index subgroups. Then $G(A,L) \cong G(\1A,\1L)$ if and only if $n = \1n$ and at least one of the following three conditions holds:
\begin{enumerate}[label=\textup{(\roman*)}]
\item \label{it:main-gen} there exists a matrix $B \in GL_n(\Z)$ such that either $\1A = BAB^{-1}$ and $\1L = BL$, or $\1A = BA^{-1}B^{-1}$ and $\1L = BAL$;
\item \label{it:main-metab} either $L = \Z^n$ or $AL = \Z^n$, either $\1L = \Z^n$ or $\1A\1L = \Z^n$, and there exists a matrix $B \in GL_n(\Q)$ such that $\1A = BA^\varepsilon B^{-1}$ for some $\varepsilon \in \{ \pm 1 \}$ and $\bigcup_{j \in \Z} A^j(\Z^n) = \bigcup_{j \in \Z} A^jB^{-1}(\Z^n)$;
\item \label{it:main-polyc} $L = AL = \1L = \1A\1L = \Z^n$, and there exists a matrix $C \in GL_{n-1}(\Z)$ of order $m < \infty$ such that $A$ and $\1A$ are conjugate in $GL_n(\Z)$ to $\begin{pmatrix} 1 & \mathbf{0}^T \\ \mathbf{u} & C \end{pmatrix}$ and $\begin{pmatrix} 1 & \mathbf{0}^T \\ \mathbf{u} & C^q \end{pmatrix}$, respectively, for some $\mathbf{u} \in \Z^{n-1}$ and some $q \in \Z$ with $\gcd(q,m) = 1$.
\end{enumerate}
\end{thm}

A Leary--Minasyan group $G(A,L)$ is an HNN-extension of $\Z^n$ with associated subgroups $L$ and $AL$, with the isomorphism $L \to AL$ given by the matrix $A$. We denote by $T(A,L)$ the Bass--Serre tree associated to this splitting. See Sections \ref{ssec:prelim-hnn} and \ref{ssec:prelim-trees} for details.

Geometrically -- that is, by considering the action of $G(A,L)$ on $T(A,L)$ -- the options in Theorem~\ref{thm:main} can be seen as follows. Option~\ref{it:main-gen} reflects the case when the vertex stabilisers and the edge stabilisers have a complete algebraic characterisation -- see Lemma~\ref{lem:large}; this happens when $G(A,L)$ is not metabelian. Option~\ref{it:main-metab} reflects the case when, even if vertex stabilisers cannot be characterised algebraically, the set of elliptic elements still can -- see Proposition~\ref{prop:elliptic}; this is the case whenever $G(A,L)$ is not polycyclic. Finally, option~\ref{it:main-polyc} reflects the remaining case, in which the set of elliptic elements may not be unique up to isomorphism; this can only happen when $G(A,L)$ is polycyclic (in fact, virtually $2$-step nilpotent -- see Remark~\ref{rmk:2step}).

In particular, Theorem~\ref{thm:main} implies that whenever two non-metabelian Leary--Minasyan groups are isomorphic, they are `obviously' isomorphic: the following result is immediate from condition~\ref{it:main-gen} in Theorem~\ref{thm:main}, see Section~\ref{sec:pfmain} for details.

\begin{cor}
Let $n,\1n \geq 0$, let $A \in GL_n(\Q)$, $\1A \in GL_{\1n}(\Q)$, and let $L \leq \Z^n \cap A^{-1}(\Z^n)$, $\1L \leq \Z^{\1n} \cap \1A^{-1}(\Z^{\1n})$ be finite index subgroups. Suppose that $G(A,L)$ is not metabelian and isomorphic to $G(\1A,\1L)$. Then there exist isomorphisms $\Phi\colon G(A,L) \to G(\1A,\1L)$ and $\xi\colon T(A,L) \to T(\1A,\1L)$ such that $\Phi(g) \cdot \xi(y) = \xi(g \cdot y)$ for all $g \in G(A,L)$ and $y \in T(A,L)$.
\end{cor}

The following can also be deduced from Theorem~\ref{thm:main}: this is immediate in the cases \ref{it:main-gen} and \ref{it:main-metab}, and is shown in Lemma~\ref{lem:conj/Q} in the case~\ref{it:main-polyc}.

\begin{cor}
Let $n,\1n \geq 0$, let $A \in GL_n(\Q)$, $\1A \in GL_{\1n}(\Q)$, and let $L \leq \Z^n \cap A^{-1}(\Z^n)$, $\1L \leq \Z^{\1n} \cap (\1A)^{-1}(\Z^{\1n})$ be finite index subgroups. Suppose that $G(A,L) \cong G(\1A,\1L)$. Then $n = \1n$, and $\1A$ is conjugate to either $A$ or $A^{-1}$ in $GL_n(\Q)$.
\end{cor}

\begin{rmk}
The cases~\ref{it:main-gen}--\ref{it:main-polyc} in Theorem~\ref{thm:main} are independent, i.~e.\ none of them covers either of the other two. Indeed, the fact that \ref{it:main-polyc} does not cover \ref{it:main-gen} or \ref{it:main-metab}, and that \ref{it:main-metab} does not cover \ref{it:main-gen}, follows by considering the values $L$ and $AL$ may take. To see that \ref{it:main-gen} does not cover \ref{it:main-metab} or \ref{it:main-polyc}, see Remark~\ref{rmk:metab} and Example~\ref{ex:polyc}, respectively. Finally, if $L = AL = \1L = \1A\1L = \Z^n$ and $\1A = BA^\varepsilon B^{-1}$ for some $\varepsilon \in \{ \pm 1 \}$, then the condition $\bigcup_{j \in \Z} A^j(\Z^n) = \bigcup_{j \in \Z} A^jB^{-1}(\Z^n)$ becomes equivalent to $\Z^n = B^{-1}(\Z^n)$, that is, $B \in GL_n(\Z)$ -- therefore, it follows from Example~\ref{ex:polyc} that \ref{it:main-metab} does not cover \ref{it:main-polyc} either.
\end{rmk}

The structure of the paper is as follows. In Section~\ref{sec:prelim}, we summarise some of the results (that we use here) on HNN-extensions (see Section~\ref{ssec:prelim-hnn}), semidirect products (see Section~\ref{ssec:semidir}), and group actions on trees (see Section~\ref{ssec:prelim-trees}). We also give a coarse classification of Leary--Minasyan groups (into non-metabelian, non-polycyclic metabelian and polycyclic groups; see Section~\ref{ssec:coarse-class}). In Sections \ref{sec:large} and \ref{sec:metab}, we give necessary conditions for an isomorphism $G(A,L) \cong G(\1A,\1L)$ to exist in the case when $G(A,L)$ is non-metabelian and non-polycyclic metabelian, respectively. In Section~\ref{sec:polyc}, we give necessary and sufficient conditions for an isomorphism $G(A,L) \cong G(\1A,\1L)$ to exist in the case when $G(A,L)$ is polycyclic, and discuss the implications of condition~\ref{it:main-polyc} in Theorem~\ref{thm:main}. In Section~\ref{sec:pfmain}, we give a proof of Theorem~\ref{thm:main} based on the results in Sections \ref{sec:large}, \ref{sec:metab} and \ref{sec:polyc}.

\begin{ack}
The author would like to thank Ian Leary and Ashot Minasyan for useful comments, and the referee for their careful reading of the paper and valuable feedback.
\end{ack}

\section{Prelimilaries and division into cases} \label{sec:prelim}

\subsection{HNN-extensions and centralisers} \label{ssec:prelim-hnn}

We first give a couple of algebraic results on HNN-extensions in general and Leary--Minasyan groups in particular.

Let $H$ be a group, let $L \leq H$ be a subgroup, and let $\theta\colon L \to H$ be an injective homomorphism. We define the \emph{HNN-extension} of $H$ with respect to $\theta$ by a (relative) presentation $G = \langle H,t \mid tht^{-1} = \theta(h) \text{ for all } h \in L \rangle$; the element $t \in G$ is called a \emph{stable letter}. A word $w = h_0 t^{\alpha_1} h_1 \cdots t^{\alpha_k} h_k$ over the alphabet $H \cup \{t\}$, where $\alpha_1,\ldots,\alpha_k \in \Z$ and $h_0,\ldots,h_k \in H$, is said to be \emph{reduced} if $\alpha_i \neq 0$ for each $i$ and if $w$ does not contain subwords of the form $tht^{-1}$ for $h \in L$ or $t^{-1}ht$ for $h \in \theta(L)$. It is clear from the presentation of $G$ (by induction on $k+\sum_{i=1}^k |\alpha_i|$, say) that every element of $g$ is represented by at least one reduced word.

\begin{prop}[Britton's Lemma, see {\cite[Theorem I.5.11]{serre}}] \label{prop:britton}
Let a group $G$ and a word $w = h_0 t^{\alpha_1} h_1 \cdots t^{\alpha_k} h_k$ be as above. If $w$ is reduced and represents the identity in $G$, then $k = 0$ and $h_0 = 1$.
\end{prop}

In particular, Proposition~\ref{prop:britton} implies that $H$ can be seen as a subgroup of $G$.

Now let $A \in GL_n(\Q)$ and let $L \leq \Z^n \cap A^{-1}(\Z^n)$ be a finite index subgroup. Consider the HNN-extension $G(A,L)$ of $H = \Z^n$ with respect to $\theta\colon L \to \Z^n, \mathbf{v} \mapsto A\mathbf{v}$, as defined in \eqref{eq:pres}. We say a word $w = \prod_{i=1}^k t^{\alpha_i} \mathbf{x}^{\mathbf{v}_i} t^{-\alpha_i}$ over $\{ \mathbf{x}^{\mathbf{v}} \mid \mathbf{v} \in \Z^n \} \cup \{t\}$ is \emph{semi-reduced} if the word $t^{\alpha_1} \mathbf{x}^{\mathbf{v}_1} t^{\alpha_2-\alpha_1} \cdots t^{\alpha_k-\alpha_{k-1}} \mathbf{x}^{\mathbf{v}_k} t^{-\alpha_k}$ is reduced.

\begin{cor} \label{cor:centralisers}
Let $\prod_{i=1}^k t^{\alpha_i} \mathbf{x}^{\mathbf{v}_i} t^{-\alpha_i}$ be a semi-reduced word representing $g \in G(A,L)$, and let $\mathbf{v} \in \Z^n$. Then $g$ commutes with $\mathbf{x}^{\mathbf{v}}$ if and only if $\mathbf{v} \in \bigcap_{j=\gamma_-+1}^{\gamma_+} A^jL$, where $\gamma_- = \min\{0,\alpha_1,\ldots,\alpha_k\}$ and $\gamma_+ = \max\{0,\alpha_1,\ldots,\alpha_k\}$.
\end{cor}

\begin{proof}
Let $\gamma \geq 0$. If $\mathbf{v} \in \bigcap_{j=1}^\gamma A^jL$, then we may show (by induction on $\gamma$, say) that $\mathbf{x}^{\mathbf{v}} = t^\gamma \mathbf{x}^{A^{-\gamma}\mathbf{v}} t^{-\gamma}$. Similarly, if $\mathbf{v} \in \bigcap_{j=-\gamma+1}^0 A^jL$, then $\mathbf{x}^{\mathbf{v}} = t^{-\gamma} \mathbf{x}^{A^\gamma\mathbf{v}} t^\gamma$. Therefore, since $\gamma_- \leq 0 \leq \gamma_+$ and since $\gamma_- \leq \alpha_i \leq \gamma_+$ for each $i$, it follows that if $\mathbf{v} \in \bigcap_{j=\gamma_-+1}^{\gamma_+} A^jL$ then $\mathbf{x}^{\mathbf{v}}$ commutes with $t^{\alpha_i} \mathbf{x}^{\mathbf{v}_i} t^{-\alpha_i}$ for each $i$; in particular, $\mathbf{x}^{\mathbf{v}}$ commutes with $g$, as required.

Conversely, suppose that $\mathbf{v} \notin \bigcap_{j=\gamma_-+1}^{\gamma_+} A^jL$. Therefore, there exists $i \in \{ 1,\ldots,k \}$ such that $\mathbf{v} \notin \bigcap_{j \in [\alpha_i]} A^jL$, where we define $[\alpha] := \{ 1,\ldots,\alpha \}$ for $\alpha \in \Z_{\geq 0}$ and $[\alpha] := \{ \alpha+1,\ldots,0 \}$ for $\alpha \in \Z_{<0}$. We choose such an $i \in \{ 1,\ldots,k \}$ to be as large as possible; it then follows from the previous paragraph that $\mathbf{x}^{\mathbf{v}}$ commutes with $t^{\alpha_j} \mathbf{x}^{\mathbf{v}_j} t^{-\alpha_j}$ for all $j > i$.

Suppose first that $\alpha_i \geq 0$, and let $\beta \geq 0$ be the largest integer such that $\mathbf{v} \in \bigcap_{j \in [\beta]} A^jL$, so that $\beta < \alpha_i$. It follows that $\mathbf{x}^{\mathbf{v}} = t^\beta \mathbf{x}^{A^{-\beta}\mathbf{v}} t^{-\beta}$, but $A^{-\beta}\mathbf{v} \notin AL$. Therefore, we have
\begin{equation} \label{eq:commutator}
g \mathbf{x}^{\mathbf{v}} g^{-1} (\mathbf{x}^{\mathbf{v}})^{-1} = t^{\alpha_1} \mathbf{x}^{\mathbf{v}_1} t^{\alpha_2-\alpha_1} \cdots \mathbf{x}^{\mathbf{v}_i} t^{\beta-\alpha_i} \mathbf{x}^{A^{-\beta}\mathbf{v}} t^{\alpha_i-\beta} \mathbf{x}^{-\mathbf{v}_i} \cdots t^{\alpha_1-\alpha_2} \mathbf{x}^{-\mathbf{v}_1} t^{-\alpha_1} \mathbf{x}^{-\mathbf{v}}.
\end{equation}
But the word on the right hand side is reduced unless $\alpha_1 = 0$, and in the latter case this word becomes reduced after replacing the terminal subword $\mathbf{x}^{-\mathbf{v}_1} \mathbf{x}^{-\mathbf{v}}$ with $\mathbf{x}^{-\mathbf{v}_1-\mathbf{v}}$. It follows by Proposition~\ref{prop:britton} that $g \mathbf{x}^{\mathbf{v}} g^{-1} (\mathbf{x}^{\mathbf{v}})^{-1} \neq 1$ in $G(A,L)$, and so $g$ does not commute with $\mathbf{x}^{\mathbf{v}}$, as required.

On the other hand, if $\alpha_i < 0$, then we set $\beta \leq 0$ be the smallest integer such that $\mathbf{v} \in \bigcap_{j \in [\beta]} A^jL$, so that $\beta > \alpha_i$. We may then show that the equation \eqref{eq:commutator} still holds and that the right hand side is reduced unless $\alpha_1 = 0$, in which case it becomes reduced after replacing $\mathbf{x}^{-\mathbf{v}_1} \mathbf{x}^{-\mathbf{v}}$ with $\mathbf{x}^{-\mathbf{v}_1-\mathbf{v}}$. It then follows, again by Proposition~\ref{prop:britton}, that $g \mathbf{x}^{\mathbf{v}} g^{-1} (\mathbf{x}^{\mathbf{v}})^{-1} \neq 1$ in $G(A,L)$, and so $g$ does not commute with $\mathbf{x}^{\mathbf{v}}$, as required.
\end{proof}

\subsection{Ascending HNN-extensions and semidirect products} \label{ssec:semidir}

Recall that a group $G$ is said to be an (\emph{internal}) \emph{semidirect product} of a normal subgroup $K \unlhd G$ and a subgroup $M \leq G$, written $G = K \rtimes M$, if we have $KM = G$ and $K \cap M = \{1\}$ (so that $G/K \cong M$). Note that in this case $M$ acts on $K$ by conjugation, and one may check that the groups $K$ and $M$ together with the $M$-action on $K$ determine the group $G$ uniquely up to isomorphism. A semidirect product $G = K \rtimes M$ is an (\emph{internal}) \emph{direct product} of $K$ and $M$, written $G = K \times M$, if $M$ is normal in $G$.

Now let $H$ be a group, and let $\theta\colon H \to H$ be an injective group homomorphism. In this case, the HNN-extension $G$ (of $H$, with respect to $\theta$ and with stable letter $t$) is said to be \emph{ascending}. We then have an chain of subgroups
\[
{}\cdots \leq tHt^{-1} \leq H \leq t^{-1}Ht \leq t^{-2}Ht^2 \leq \cdots{}
\]
of $G$, implying that $\bigcup_{j \in \Z} t^jHt^{-j}$ is a subgroup of $G$ (which is invariant under conjugation by $t$ and so normal). One may then verify the following well-known description of ascending HNN-extensions.

\begin{lem} \label{lem:ascHNN}
Let $G$ be the ascending HNN-extension of a group $H$ with respect to an injective group homomorphism $\theta\colon H \to H$, with stable letter $t$. Then $G = K \rtimes \langle t \rangle$, where $K = \bigcup_{j \in \Z} t^jHt^{-j}$. \qed
\end{lem}

\subsection{Actions on trees} \label{ssec:prelim-trees}

HNN-extensions fit into a more general geometric setting of groups acting on (simplicial) trees without global fixed points. We briefly describe some of the main properties of group actions on trees, and refer the interested reader to \cite{serre} or \cite{culler} for a more thorough description.

For a simplicial tree $T$, we write $V(T)$ and $E(T)$ for the sets of vertices and edges of $T$, respectively. Given a group $G$ acting on a simplicial tree $T$ by automorphisms, we may define the \emph{translation length} function $\tau_T\colon G \to \Z$ by $\tau_T(g) := \min \{ d_T(x,gx) \mid x \in T \}$. An element $g \in G$ is said to be \emph{elliptic} (with respect to the action of $G$ on $T$) if $\tau_T(g) = 0$, and \emph{hyperbolic} otherwise. The following result is well-known.

\begin{prop}[see {\cite[1.3(ii)]{culler}}] \label{prop:axis}
If an element $g \in G$ is hyperbolic (with respect to an action on a tree $T$), then the set $\ell = \{ x \in T \mid d_T(x,gx) = \tau_T(g) \}$ is a $\langle g \rangle$-invariant subtree of $T$ isometric to $\R$, and $g$ acts on $\ell$ as a translation by $\tau_T(g)$.
\end{prop}

The subset $\ell$ as in Proposition~\ref{prop:axis} is called the \emph{axis} of $g$.

When $G$ is the HNN-extension of a group $H$ with respect to $\theta\colon L \to H$ (for some $L \leq H$) with stable letter $t \in G$, there is a canonical tree $T$ with a $G$-action, called the \emph{Bass--Serre tree} of $G$: the vertices of $T$ are the left cosets $gH$ of $H$ in $G$, the edges of $T$ are the left cosets $gL$ of $L$ in $G$, and a given edge $gL$ of $T$ is incident to the vertices $gH$ and $gt^{-1}H$ of $T$. Then $G$ acts on $T$ by left multiplication, and the fact that $T$ is a tree can be seen as a geometric interpretation of Britton's Lemma (Proposition~\ref{prop:britton}). It is clear that $G$ acts on $T$ with one orbit of vertices, one orbit of edges, and without edge inversions, and that the stabilisers of vertices and edges of $T$ are precisely the $G$-conjugates of $H$ and of $L$, respectively.

In the case when $G = G(A,L)$ is the HNN-extension of $H = \Z^n$ with respect to the map $\theta\colon L \to H, \mathbf{v} \mapsto A\mathbf{v}$, we denote by $T = T(A,L)$ the corresponding Bass--Serre tree. One can see from this description that if $\mathcal{E} \subseteq E(T)$ is the set of edges starting at the vertex $H \in V(T)$, then $\mathcal{E} = \mathcal{E}_1 \sqcup \mathcal{E}_2$, where $\mathcal{E}_1$ is a collection of $[\Z^n:L]$ edges each with stabiliser $L$, and $\mathcal{E}_2$ is a collection of $[\Z^n:AL]$ edges each with stabiliser $AL$; moreover, given $e,e' \in \mathcal{E}$, an element $g \in H$ with $e' = ge$ exists if and only if either $e,e' \in \mathcal{E}_1$ or $e,e' \in \mathcal{E}_2$. These observations imply the following result.

\begin{lem} \label{lem:stabilisers}
For $G = G(A,L)$ and $T = T(A,L)$, the following hold.
\begin{enumerate}[label=\textup{(\roman*)}]
\item \label{it:stab-slidefree} We have $L \nsubseteq AL$ and $AL \nsubseteq L$ if and only if given any two edges $e,e' \in E(T)$ starting at a vertex $v \in V(T)$ such that $\Stab_G(e) \subseteq \Stab_G(e')$, there exists $g \in \Stab_G(x)$ such that $ge = e'$.
\item \label{it:stab-reduced} We have $L \neq \Z^n \neq AL$ if and only if $\Stab_G(e) \subsetneq \Stab_G(v)$ for every edge $e \in E(T)$ incident to a vertex $v \in V(T)$.
\qed
\end{enumerate}
\end{lem}

\subsection{Coarse classification of Leary--Minasyan groups} \label{ssec:coarse-class}

We start the isomorphism classification by giving an algebraic characterisation of polycyclic and metabelian Leary--Minasyan groups. This will allow us to split our argument into polycyclic, non-polycyclic metabelian, and non-metabelian cases. In the latter two cases, we also give an algebraic description of elliptic elements of $G(A,L)$: see Proposition~\ref{prop:elliptic}.

Recall that a group $G$ is said to be \emph{metabelian} if there exists a normal subgroup $K \unlhd G$ with $K$ and $G/K$ abelian. Recall also that a group $G$ is \emph{polycyclic} if it has a subnormal series
\begin{equation} \label{eq:subnormal}
\{ 1 \} = G_0 \leq G_1 \leq \cdots \leq G_r = G
\end{equation}
with the quotient $G_i/G_{i-1}$ cyclic for each $i$. The \emph{Hirsch length} of a polycyclic group $G$ as above is the number of $i \in \{ 1,\ldots,r \}$ such that $G_i/G_{i-1} \cong \Z$; it is well-known that the Hirsch length of $G$ does not depend on the choice of subnormal series as in \eqref{eq:subnormal}: see \cite[Exercise 1.8]{segal}.

\begin{lem} \label{lem:polycyclic}
Let $A \in GL_n(\Q)$ and $L \leq \Z^n \cap A^{-1}(\Z^n)$ be as above. Then $G(A,L)$ is polycyclic if and only if $L = AL = \Z^n$ (in which case it has Hirsch length $n+1$), and metabelian if and only if either $L = \Z^n$ or $AL = \Z^n$.
\end{lem}

\begin{proof}
Let $G = G(A,L)$, and let $H = \langle x_1,\ldots,x_n \rangle < G$.

If $L = AL = \Z^n$, then the map $H \to H, h \mapsto tht^{-1}$ is an isomorphism and so $G = H \rtimes \langle t \rangle$ by Lemma~\ref{lem:ascHNN}. Thus $G$ is polycyclic of Hirsch length $n+1$.

If $L = \Z^n \neq AL$, then $G$ is a strictly ascending HNN-extension of $H$, implying that we have a strictly ascending chain of subgroups $H \lneq t^{-1}Ht \lneq t^{-2}Ht^2 \lneq \cdots$ of $G$; thus $G$ cannot be polycyclic: see \cite[Proposition~1.4]{segal}. On the other hand, $G = K \rtimes \langle t \rangle$ by Lemma~\ref{lem:ascHNN}, where $K = \bigcup_{j \in \Z} t^jHt^{-j}$ is the union of an ascending chain of abelian subgroups and so abelian; thus $G$ is metabelian. A similar argument shows that $G$ is metabelian but not polycyclic whenever $AL = \Z^n \neq L$.

Finally, suppose that $L \neq \Z^n \neq AL$. Then $G$ is an HNN-extension of $H$ whose associated subgroups $H_1 = \langle \mathbf{x}^{\mathbf{v}} \mid \mathbf{v} \in L \}$ and $H_2 = \langle \mathbf{x}^{\mathbf{v}} \mid \mathbf{v} \in AL \}$ are both proper in $H$. In this case, it follows from Proposition~\ref{prop:britton} that $t$ and $h_2th_1$ generate a non-abelian free subgroup of $G$ for any $h_1 \in H \setminus H_1$ and $h_2 \in H \setminus H_2$; in particular, $G$ is not metabelian.
\end{proof}

\begin{prop} \label{prop:elliptic}
Let $A \in GL_n(\Q)$ and $L \leq \Z^n \cap A^{-1}(\Z^n)$ be as above, with either $L \neq \Z^n$ or $AL \neq \Z^n$, and let $g \in G := G(A,L)$. Then $g$ is elliptic (with respect to the action of $G$ on $T(A,L)$) if and only if $\langle g \rangle \cap C_G(h)$ is non-trivial for every $h \in G'$.
\end{prop}

\begin{proof}
Suppose first that $g \in G$ is elliptic. By replacing $g$ with its conjugate if necessary, we may assume that $g \in \langle x_1,\ldots,x_n \rangle$, that is, $g = \mathbf{x}^{\mathbf{w}}$ for some $\mathbf{w} \in \Z^n$. Let $h \in G'$, and note that we have a homomorphism $\psi\colon G \to \Z$ sending $t \mapsto 1$ and $x_i \mapsto 0$ for all $i$. As $\Z$ is abelian, it follows that $\psi(h) = 0$, and so we have $h = \prod_{i=1}^k t^{\alpha_i} \mathbf{x}^{\mathbf{v}_i} t^{-\alpha_i}$ for some $\alpha_1,\ldots,\alpha_k \in \Z$ and $\mathbf{v}_1,\ldots,\mathbf{v}_k \in \Z^n$. Let $\gamma = \max \{ |\alpha_1|,\ldots,|\alpha_k| \}$.

Now note that, since $A^j \in GL_n(\Q)$ for all $j \in \Z$, the group $A^jL \cap \Z^n$ is a finite-index subgroup of $\Z^n$ for each $j \in \Z$. Therefore, $\bigcap_{j=-\gamma+1}^\gamma A^jL$ has finite index in $\Z^n$ (as it is an intersection of finitely many finite-index subgroups), and so there exists $m \in \Z_{\geq 1}$ such that $m\mathbf{w} \in \bigcap_{j=-\gamma+1}^\gamma A^jL$. Then $g^m = \mathbf{x}^{m\mathbf{w}} \in C_G(h)$ by Corollary~\ref{cor:centralisers}, and so $\langle g \rangle \cap C_G(h)$ is non-trivial, as required.

Conversely, let $g \in G$ be hyperbolic, and let $\ell \subseteq T(A,L)$ be the axis of $g$, as described in Proposition~\ref{prop:axis}. Suppose for contradiction that $\langle g \rangle \cap C_G(h)$ is non-trivial for each $h \in G'$. Then, given $h \in G'$, there exists $m \geq 1$ such that $g^m$ commutes with $h$. Now note that $\ell$ (respectively $h\ell$) is the axis of $g^m$ (respectively $hg^mh^{-1}$), so since $g^m = hg^mh^{-1}$ it follows by the uniqueness of axes of hyperbolic elements that $h\ell = \ell$. In particular, since this is true for any $h \in G'$, it follows that $\ell$ is invariant under the action of $G'$.

But this implies that we have a group homomorphism $\Phi\colon G' \to \Aut(\ell) \cong D_\infty$, where $\Aut(\ell)$ is the group of graph automorphisms of $\ell$, and therefore $G'/\ker(\Phi)$ is isomorphic to a subgroup of the infinite dihedral group $D_\infty$. Moreover, given any vertex $x \in \ell$ we have $\ker(\Phi) \leq \Stab_{G'}(x) \leq \Stab_G(x) \cong \Z^n$. As $G$ is finitely generated, we also know that $G/G'$ is finitely generated abelian. It follows that we have a subnormal series $1 \unlhd \ker(\Phi) \unlhd G' \unlhd G$, with each quotient of consecutive terms polycyclic. This implies that $G$ is polycyclic, contradicting Lemma~\ref{lem:polycyclic}.

Thus $\langle g \rangle \cap C_G(h)$ must be trivial for some $h \in G'$, as required.
\end{proof}

\section{Non-metabelian groups} \label{sec:large}

In this section, we give a few results on non-metabelian Leary--Minasyan groups; by Lemma~\ref{lem:polycyclic}, we know that $G(A,L)$ is not metabelian precisely when $L \neq \Z^n \neq AL$. We use the following result to obtain algebraic information about vertex and edge stabilisers under the action of $G(A,L)$ on $T(A,L)$.

In the following result, we say a subgroup $H_0 < G(A,L)$ is \emph{elliptic} if its action on $T(A,L)$ has a global fixed point. Thus any elliptic subgroup consists of elliptic elements; the converse is true when the subgroup is finitely generated: see \cite[Corollary~3 on p.~65]{serre}.

\begin{lem} \label{lem:large}
Let $A \in GL_n(\Q)$ and $L \leq \Z^n \cap A^{-1}(\Z^n)$ be as above, with $L \neq \Z^n$ and $AL \neq \Z^n$, and let $H = \langle x_1,\ldots,x_n \rangle < G := G(A,L)$. Then the following hold.
\begin{enumerate}[label=\textup{(\roman*)}]
\item \label{it:large-H} $H$ is a maximal elliptic subgroup, unique such up to conjugation.
\item \label{it:large-psi} The map sending $x_i \mapsto 0$ (for $1 \leq i \leq n$) and $t \mapsto 1$ extends to a surjective group homomorphism $\psi\colon G \to \Z$ with kernel $K := \langle\!\langle H \rangle\!\rangle$. Moreover, for each $g \in G$ there exists an integer $m = m(g) \geq 1$ such that $g \mathbf{x}^{\mathbf{v}} g^{-1} = \mathbf{x}^{A^{\psi(g)}\mathbf{v}}$ for all $\mathbf{v} \in m\Z^n$.
\item \label{it:large-H1} The set $H_1 = \{ g \in H \mid C_G(g) \neq H \}$ is equal to $\{ \mathbf{x}^{\mathbf{v}} \mid \mathbf{v} \in AL \cup L \}$. In particular, $H_1$ is a subgroup of $H$ if and only if either $AL \subseteq L$ or $L \subseteq AL$.
\item \label{it:large-H1c} We have $AL = L$ if and only if $H_1 \subseteq Z(K)$.
\end{enumerate}
\end{lem}

\begin{proof}
Let $T = T(A,L)$ be the Bass--Serre tree corresponding to the splitting of $G = G(A,L)$ as an HNN-extension, and let $a \in V(T)$ be such that $H = \Stab_G(a)$.

\begin{enumerate}[label=(\roman*)]

\item Since $H$ stabilises $a \in V(T)$, it is elliptic. To show maximality, suppose for contradiction that there exists an element $g \in G \setminus H$ such that the subgroup $\widetilde{H} = \langle H,g \rangle$ is elliptic. Let $\widetilde{a} \in V(T)$ be a vertex stabilised by $\widetilde{H}$; since $\Stab_G(a) = H \subsetneq \widetilde{H}$, we have $\widetilde{a} \neq a$. Now if $e \subseteq T$ is the first edge of $T$ on the geodesic from $a$ to $\widetilde{a}$, then $H$ stabilises $e$ (as it stabilises both $a$ and $\widetilde{a}$), and so $\Stab_G(e) \subseteq \Stab_G(a) = H \subseteq \Stab_G(e)$, implying that $\Stab_G(e) = \Stab_G(a)$; by Lemma~\ref{lem:stabilisers}\ref{it:stab-reduced}, this contradicts the fact that $L \neq \Z^n$ and $AL \neq \Z^n$. Thus $H$ must be a maximal elliptic subgroup, as required.

To show uniqueness, note that any maximal elliptic subgroup must be a vertex stabiliser under the action of $G$ on $T$. But since this action is transitive on vertices, it follows that any such subgroup is a conjugate of $H$.

\item The fact that $\psi\colon G \to \Z$ is a well-defined surjective homomorphism with kernel $\langle\!\langle H \rangle\!\rangle$ follows directly from the presentation \eqref{eq:pres} of $G = G(A,L)$.

As $L$ is a finite-index subgroup of $\Z^n$ and as $A \in GL_n(\Q)$, the group $A^jL \cap \Z^n$ is a finite-index subgroup of $\Z^n$ for any $j \in \Z$, and hence so is the intersection $\bigcap_{j=-\gamma+1}^\gamma A^jL$ for any $\gamma \in \Z_{\geq 0}$: therefore, there exists $m = m(\gamma) \in \Z_{\geq 1}$ such that $m\Z^n \subseteq \bigcap_{j=-\gamma+1}^\gamma A^jL$. Now let $g \in G$: we can write $g = t^{\psi(g)} \cdot \prod_{i=1}^k t^{\alpha_i} \mathbf{x}^{\mathbf{v}_i} t^{-\alpha_i}$ for some $\alpha_1,\ldots,\alpha_k \in \Z$ and $\mathbf{v}_1,\ldots,\mathbf{v}_k \in \Z^n$, where the word $\prod_{i=1}^k t^{\alpha_i} \mathbf{x}^{\mathbf{v}_i} t^{-\alpha_i}$ is semi-reduced. Let $\gamma = \max \{ |\psi(g)|, |\alpha_1|, \ldots, |\alpha_k| \}$, and let $m = m(\gamma) \in \Z_{\geq 1}$ be as above. Given $\mathbf{v} \in m\Z^n$, it then follows from Corollary~\ref{cor:centralisers} that $\mathbf{x}^{\mathbf{v}}$ commutes with $\prod_{i=1}^k t^{\alpha_i} \mathbf{x}^{\mathbf{v}_i} t^{-\alpha_i}$. Moreover, $t^{\psi(g)} \mathbf{x}^{\mathbf{v}} t^{-\psi(g)} = \mathbf{x}^{A^{\psi(g)}\mathbf{v}}$: we can show this by induction on $|\psi(g)|$, see the proof of Corollary~\ref{cor:centralisers}. Therefore,
\begin{align*}
g \mathbf{x}^{\mathbf{v}} g^{-1} &= t^{\psi(g)} \cdot \left( \prod_{i=1}^k t^{\alpha_i} \mathbf{x}^{\mathbf{v}_i} t^{-\alpha_i} \right) \cdot \mathbf{x}^{\mathbf{v}} \cdot \left( \prod_{i=1}^k t^{\alpha_i} \mathbf{x}^{\mathbf{v}_i} t^{-\alpha_i} \right)^{-1} \cdot t^{-\psi(g)} \\ &= t^{\psi(g)} \mathbf{x}^{\mathbf{v}} t^{-\psi(g)} = \mathbf{x}^{A^{\psi(g)}\mathbf{v}}
\end{align*}
for all $\mathbf{v} \in m\Z^n$, as required.

\item As $H$ is abelian, we have $C_G(h) \supseteq H$ for all $h \in H$. Therefore, $H_1$ consists of precisely those elements of $H$ whose centralisers contain $H$ as a proper subgroup. Now given $h \in H$, the centraliser $C_G(h)$ leaves the fixed point set of $h$ invariant, and so if $C_G(h) \supsetneq H$ then the fixed point set of $h$ must properly contain $\{a\}$. Conversely, suppose that $hb = b$ for some $b \in V(T) \setminus \{a\}$. Then $b = ga$ for some $g \in G \setminus H$ since the action $G \curvearrowright T$ is transitive on vertices, and so $h \in gHg^{-1}$; therefore, $gHg^{-1} \subseteq C_G(h)$ since $gHg^{-1}$ is abelian. On the other hand, since $L \neq \Z^n \neq AL$, Lemma~\ref{lem:stabilisers}\ref{it:stab-reduced} implies that there exists $h_0 \in \Stab_G(b) = gHg^{-1}$ such that $h_0a \neq a$, and so $h_0 \notin H$: thus $H$ must be a proper subgroup of $C_G(h)$.

This argument shows that a given $h \in H$ is an element of $H_1$ if and only if $hb = b$ for some $b \in V(T) \setminus \{a\}$, which happens if and only if $he = e$ for some edge $e$ starting at $a$. But the stabilisers of edges of $T$ starting at $a$ are precisely $\{ \mathbf{x}^{\mathbf{v}} \mid \mathbf{v} \in L \}$ and $\{ \mathbf{x}^{\mathbf{v}} \mid \mathbf{v} \in AL \}$, implying that $H_1 = \{ \mathbf{x}^{\mathbf{v}} \mid \mathbf{v} \in AL \cup L \}$, as required.

For the second part, note that if $AL \subseteq L$ then $H_1 = \{ \mathbf{x}^{\mathbf{v}} \mid \mathbf{v} \in L \}$, whereas if $L \subseteq AL$ then $H_1 = \{ \mathbf{x}^{\mathbf{v}} \mid \mathbf{v} \in AL \}$ -- so in both cases $H_1$ is a subgroup of $H$. Otherwise, there exist elements $\mathbf{v} \in AL \setminus L$ and $\mathbf{w} \in L \setminus AL$, and we have $\mathbf{v}+\mathbf{w} \notin L \cup AL$; therefore, $\mathbf{x}^{\mathbf{v}},\mathbf{x}^{\mathbf{w}} \in H_1$ but $\mathbf{x}^{\mathbf{v}}\mathbf{x}^{\mathbf{w}} = \mathbf{x}^{\mathbf{v}+\mathbf{w}} \notin H_1$, and so $H_1$ is not a subgroup of $H$.

\item By part~\ref{it:large-H1}, we know that $H_1 = \{ \mathbf{x}^{\mathbf{v}} \mid \mathbf{v} \in AL \cup L \}$.

Suppose first that $L = AL$, and let $\mathbf{x}^{\mathbf{v}} \in H_1$, so that $\mathbf{v} \in L$. Since $AL = L$, we have $A^jL = L$ for each $j \in \Z$, and so $\mathbf{v} \in L = \bigcap_{j=-\infty}^\infty A^jL$. It then follows from Corollary~\ref{cor:centralisers} that $\mathbf{x}^{\mathbf{v}} \in Z(K)$: therefore, $H_1 \subseteq Z(K)$, as required.

Conversely, suppose that $L \neq AL$, and let $\mathbf{v} \in (L \cup AL) \setminus (L \cap AL)$. Then $\mathbf{x}^{\mathbf{v}} \in H_1$. However, if $\mathbf{v} \notin L$ then for any $\mathbf{w} \in \Z^n \setminus AL$ we have $\mathbf{x}^{\mathbf{v}} \notin C_K(t^{-1} \mathbf{x}^{\mathbf{w}} t)$ by Corollary~\ref{cor:centralisers}; similarly, if $\mathbf{v} \notin AL$ then $\mathbf{x}^{\mathbf{v}} \notin C_K(t \mathbf{x}^{\mathbf{w}} t^{-1})$ for any $\mathbf{w} \in \Z^n \setminus L$. Thus in either case $\mathbf{x}^{\mathbf{v}} \notin Z(K)$ and so $H_1 \nsubseteq Z(K)$, as required. \qedhere

\end{enumerate}
\end{proof}

The following result can be extracted from Theorem~1.2 and the implication (d)~$\Rightarrow$~(a) of Theorem~1.1 in \cite{forester}.

\begin{thm}[M.~Forester {\cite[Theorems 1.1 and 1.2]{forester}}] \label{thm:forester}
Let $G$ be a group acting cocompactly on simplicial trees $T$ and $\1T$ by automorphisms and without edge inversions. Suppose that the sets of elliptic elements of $G$ with respect to the two actions coincide. Moreover, suppose that
\begin{enumerate}[label=\textup{(\roman*)}]
\item \label{it:forester-slidefree} for any edges $e,f \subseteq T$ starting at a vertex $v \in V(T)$, if $\Stab_G(e) \subseteq \Stab_G(f)$ then $ge = f$ for some $g \in \Stab_G(v)$; and
\item \label{it:forester-reduced} for any edge $\1e \subseteq \1T$ incident to a vertex $\1v \in V(\1T)$, we have $\Stab_G(\1e) \subsetneq \Stab_G(\1v)$.
\end{enumerate}
Then there is a unique $G$-equivariant isomorphism $T \to \1T$.
\end{thm}

We can use Lemma~\ref{lem:large} and Theorem~\ref{thm:forester} to show the following result, which proves the `only if' direction of Theorem~\ref{thm:main} in the case when $G(A,L)$ is not metabelian.

\begin{prop} \label{prop:main-large}
Let $n,\1n \geq 0$, let $A \in GL_n(\Q)$, $\1A \in GL_{\1n}(\Q)$, and let $L \leq \Z^n \cap A^{-1}(\Z^n)$, $\1L \leq \Z^{\1n} \cap \1A^{-1}(\Z^{\1n})$ be finite index subgroups. Suppose that $G(A,L) \cong G(\1A,\1L)$, and that $L \neq \Z^n \neq AL$ and $\1L \neq \Z^{\1n} \neq \1A\1L$. Then $n = \1n$, and there exists $B \in GL_n(\Z)$ such that either $\1A = BAB^{-1}$ and $\1L = BL$, or $\1A = BA^{-1}B^{-1}$ and $\1L = BAL$.
\end{prop}

\begin{proof}
As in \eqref{eq:pres}, we fix presentations
\begin{align*}
G(A,L) &= \langle x_1,\ldots,x_n,t \mid [x_i,x_j]=1 \text{ for } 1 \leq i < j \leq n, t\mathbf{x}^{\mathbf{v}}t^{-1} = \mathbf{x}^{A\mathbf{v}} \text{ for } \mathbf{v} \in L \rangle
\shortintertext{and}
G(\1A,\1L) &= \langle \1x_1,\ldots,\1x_{\1n},\1t \mid [\1x_i,\1x_j]=1 \text{ for } 1 \leq i < j \leq \1n, \1t\1{\mathbf{x}}^{\1{\mathbf{v}}}\1t^{-1} = \1{\mathbf{x}}^{\1A\1{\mathbf{v}}} \text{ for } \1{\mathbf{v}} \in \1L \rangle,
\end{align*}
where we write $\mathbf{x}^{\mathbf{w}} := x_1^{w_1} \cdots x_n^{w_n}$ for $\mathbf{w} = (w_1,\ldots,w_n) \in \Z^n$ and $\1{\mathbf{x}}^{\1{\mathbf{w}}} := \1x_1^{\1w_1} \cdots \1x_{\1n}^{\1w_{\1n}}$ for $\1{\mathbf{w}} = (\1w_1,\ldots,\1w_{\1n}) \in \Z^{\1n}$. We also fix subgroups $H = \langle x_1,\ldots,x_n \rangle < G(A,L)$ and $\1H = \langle \1x_1, \ldots, \1x_{\1n} \rangle < G(\1A,\1L)$, and an isomorphism $\Phi\colon G(A,L) \to G(\1A,\1L)$.

By Proposition~\ref{prop:elliptic}, the set of elliptic elements in $G(A,L)$ with respect to the action on $T(A,L)$ has a purely algebraic characterisation. In particular, if $g \in G(A,L)$ is an elliptic (respectively hyperbolic) element with respect to $G(A,L) \curvearrowright T(A,L)$, then $\Phi(g) \in G(\1A,\1L)$ must be an elliptic (respectively hyperbolic) element with respect to $G(\1A,\1L) \curvearrowright T(\1A,\1L)$. Since finitely generated elliptic subgroups (of a group acting on a tree) are precisely the finitely generated subgroups consisting of elliptic elements \cite[Corollary~3 on p.~65]{serre}, it follows that a given finitely generated subgroup $\widetilde{H} < G(A,L)$ is elliptic with respect to $G(A,L) \curvearrowright T(A,L)$ if and only if $\Phi(\widetilde{H}) < G(\1A,\1L)$ is elliptic with respect to $G(\1A,\1L) \curvearrowright T(\1A,\1L)$. Now by Lemma~\ref{lem:large}\ref{it:large-H}, $H$ and $\1H$ are maximal elliptic subgroups of $G(A,L)$ and $G(\1A,\1L)$, respectively. It follows that $\Phi(H)$ and $\1H$ are maximal elliptic subgroups of $G(\1A,\1L)$ with respect to the action on $T(\1A,\1L)$, and so, again by Lemma~\ref{lem:large}\ref{it:large-H}, $\Phi(H)$ and $\1H$ are conjugate in $G(\1A,\1L)$.

By composing an inner automorphism of $G(\1A,\1L)$ with $\Phi$, we may therefore assume that $\Phi(H) = \1H$. In particular, this implies that $n = \1n$, and that there exists matrix $B \in GL_n(\Z)$ such that $\Phi(\mathbf{x}^{\mathbf{v}}) = \1{\mathbf{x}}^{B\mathbf{v}}$ for all $\mathbf{v} \in \Z^n$.

Now let $K$ be the normal closure of $H$ in $G(A,L)$, so that $\1K := \Phi(K)$ is the normal closure of $\1H$ in $G(\1A,\1L)$. By Lemma~\ref{lem:large}\ref{it:large-psi}, there exist surjections $\psi\colon G(A,L) \to \Z$ and $\1\psi\colon G(\1A,\1L) \to \Z$ with kernels $K$ and $\1K$, respectively, such that for any $g \in G(A,L)$ there exist constants $m,\1m \geq 1$ such that $g\mathbf{x}^{\mathbf{v}}g^{-1} = \mathbf{x}^{A^{\psi(g)}\mathbf{v}}$ for all $\mathbf{v} \in m\Z^n$ and $\Phi(g)\1{\mathbf{x}}^{\1{\mathbf{v}}}\Phi(g)^{-1} = \1{\mathbf{x}}^{\1A^{\1\psi(\Phi(g))}\1{\mathbf{v}}}$ for all $\1{\mathbf{v}} \in \1m\Z^n$. Since $\1\psi \circ \Phi$ and $\psi$ are both surjections $G(A,L) \to \Z$ with kernel $K$, there exists $\varepsilon \in \{ \pm 1 \}$ such that $\1\psi(\Phi(g)) = \varepsilon \psi(g)$ for all $g \in G(A,L)$.

By taking $M = \lcm(m,\1m)$, it follows that, for any $g \in G(A,L)$ and $\mathbf{v} \in M\Z^n$,
\[
\1{\mathbf{x}}^{BA^{\psi(g)}\mathbf{v}} = \Phi(\mathbf{x}^{A^{\psi(g)}\mathbf{v}}) = \Phi(g\mathbf{x}^{\mathbf{v}}g^{-1}) = \Phi(g)\1{\mathbf{x}}^{B\mathbf{v}}\Phi(g)^{-1} = \1{\mathbf{x}}^{\1A^{\1\psi(\Phi(g))}B\mathbf{v}} = \1{\mathbf{x}}^{\1A^{\varepsilon\psi(g)}B\mathbf{v}}.
\]
Therefore, $\1A^{\varepsilon\psi(g)}B = BA^{\psi(g)}$ for all $g \in G(A,L)$, and so $\1A = B A^\varepsilon B^{-1}$ since the map $\psi\colon G(A,L) \to \Z$ is surjective.

Let $H_1 = \{ g \in H \mid C_{G(A,L)}(g) \neq H \}$ and let $\1H_1 = \{ g \in \1H \mid C_{G(\1A,\1L)}(g) \neq \1H \}$. It then follows that $\Phi(H_1) = \1H_1$ and, by Lemma~\ref{lem:large}\ref{it:large-H1}, we have $H_1 = \{ \mathbf{x}^{\mathbf{v}} \mid \mathbf{v} \in AL \cup L \}$ and $\1H_1 = \{ \1{\mathbf{x}}^{\1{\mathbf{v}}} \mid \1{\mathbf{v}} \in \1A\1L \cup \1L \}$. Since $\Phi(\mathbf{x}^{\mathbf{v}}) = \1{\mathbf{x}}^{B\mathbf{v}}$ for all $\mathbf{v} \in \Z^n$, it follows that $\1A\1L \cup \1L = B(AL \cup L)$. We now consider three cases:

\begin{description}

\item[If $L \nsubseteq AL$ and $AL \nsubseteq L$] Let $G = G(A,L)$, $T = T(A,L)$ and $\1T = T(\1A,\1L)$, and note that $\1T$ can be seen as a $G$-tree via the isomorphism $\Phi$. Then, by Lemma~\ref{lem:stabilisers}\ref{it:stab-slidefree}, the assumption that $L \nsubseteq AL$ and $AL \nsubseteq L$ implies that the condition~\ref{it:forester-slidefree} in Theorem~\ref{thm:forester} is satisfied. Moreover, since $\1L \neq \Z^n \neq \1A\1L$, it follows from Lemma~\ref{lem:stabilisers}\ref{it:stab-reduced} that the condition~\ref{it:forester-reduced} in Theorem~\ref{thm:forester} is satisfied. Therefore, by Theorem~\ref{thm:forester} there exists an isomorphism $\xi\colon T \to \1T$ such that $\xi(g \cdot y) = \Phi(g) \cdot \xi(y)$ for all $g \in G$ and $y \in T$.

Now let $y_0 \in V(T)$ be the vertex with $\Stab_G(y_0) = H$. Then the vertex $\Phi(t) \cdot \xi(y_0) = \xi(t \cdot y_0)$ is adjacent to $\xi(y_0)$ in $\1T$ since $t \cdot y_0$ is adjacent to $y_0$ in $T$. Since $\Stab_{G(\1A,\1L)}(\xi(y_0)) = \Phi(H) = \1H$, this implies that $\Phi(t) = \1{\mathbf{x}}^{\mathbf{v}} \1t^\delta \1{\mathbf{x}}^{\mathbf{w}}$ for some $\mathbf{v},\mathbf{w} \in \Z^n$ and $\delta \in \{ \pm 1 \}$; moreover, we have
\[
\varepsilon = \varepsilon\psi(t) = \1\psi(\Phi(t)) = \1\psi(\1{\mathbf{x}}^{\mathbf{v}} \1t^\delta \1{\mathbf{x}}^{\mathbf{w}}) = \delta\1\psi(\1t) = \delta,
\]
where $\varepsilon \in \{ \pm 1 \}$ is as above, and so $\Phi(t) = \1{\mathbf{x}}^{\mathbf{v}} \1t^\varepsilon \1{\mathbf{x}}^{\mathbf{w}}$.

Suppose first that $\varepsilon = 1$. Then $\1A = BAB^{-1}$ as shown above. Moreover, let $\mathbf{u} \in \Z^n$. Then we have $\mathbf{u} \in L$ if and only if $t \mathbf{x}^{\mathbf{u}} t^{-1} \in H$ (by Proposition~\ref{prop:britton}), if and only if $t \mathbf{x}^{\mathbf{u}} t^{-1} \cdot y_0 = y_0$, if and only if $\Phi(t\mathbf{x}^{\mathbf{u}}t^{-1}) \cdot \xi(y_0) = \xi(t \mathbf{x}^{\mathbf{u}} t^{-1} \cdot y_0) = \xi(y_0)$, if and only if $\Phi(t\mathbf{x}^{\mathbf{u}}t^{-1}) \in \1H$. But we can calculate that $\Phi(t\mathbf{x}^{\mathbf{u}}t^{-1}) = (\1{\mathbf{x}}^{\mathbf{v}} \1t \1{\mathbf{x}}^{\mathbf{w}}) \1{\mathbf{x}}^{B\mathbf{u}} (\1{\mathbf{x}}^{\mathbf{v}} \1t \1{\mathbf{x}}^{\mathbf{w}})^{-1} = \1{\mathbf{x}}^{\mathbf{v}} \1t \1{\mathbf{x}}^{B\mathbf{u}} \1t^{-1} \1{\mathbf{x}}^{-\mathbf{v}}$, which is in $\1H$ if and only if $B\mathbf{u} \in \1L$ (by Proposition~\ref{prop:britton}). Thus $\mathbf{u} \in L$ if and only if $B\mathbf{u} \in \1L$, and so $\1L = BL$, as required.

Suppose now that $\varepsilon = -1$. Then $\1A = BA^{-1}B^{-1}$ and, similarly to the previous case, for a given $\mathbf{u} \in \Z^n$ we have $\mathbf{u} \in AL$ if and only if $\Phi(t^{-1}\mathbf{x}^{\mathbf{u}}t) \in \1H$. But we can calculate that $\Phi(t^{-1}\mathbf{x}^{\mathbf{u}}t) = \1{\mathbf{x}}^{-\mathbf{w}} \1t \1{\mathbf{x}}^{B\mathbf{u}} \1t^{-1} \1{\mathbf{x}}^{\mathbf{w}}$, which is in $\1H$ if and only if $B\mathbf{u} \in \1L$ (again by Proposition~\ref{prop:britton}). Thus $\mathbf{u} \in AL$ if and only if $B\mathbf{u} \in \1L$, and so $\1L = BAL$, as required.

\item[If $AL = L$] Then we have $H_1 \subseteq Z(K)$ by Lemma~\ref{lem:large}\ref{it:large-H1c}. Since $\Phi$ is an isomorphism, this implies that $\1H_1 = \Phi(H_1) \subseteq Z(\Phi(K)) = Z(\1K)$, and so we have $\1A \1L = \1L$ by Lemma~\ref{lem:large}\ref{it:large-H1c}. In particular, we have $\1L = \1A\1L \cup \1L = B(AL \cup L) = BL = BAL$; combined with the equality $\1A = BA^\varepsilon B^{-1}$, this gives the result.

\item[If $AL \subsetneq L$ or $L \subsetneq AL$] Then by parts \ref{it:large-H1} and \ref{it:large-H1c} of Lemma~\ref{lem:large}, $H_1$ is a group not contained in $Z(K)$. Since $\1H_1 = \Phi(H_1)$ and $\1K = \Phi(K)$, it follows that $\1H_1$ is a group not contained in $Z(\1K)$, and so (again by parts \ref{it:large-H1} and \ref{it:large-H1c} of Lemma~\ref{lem:large}) we have either $\1A\1L \subsetneq \1L$ or $\1L \subsetneq \1A\1L$.

Now note that, since $L$ and $AL$ have finite index in $\Z^n$, we have $\lvert\det(A)\rvert = [L:AL] > 1$ if $AL \subsetneq L$, and $\lvert\det(A)\rvert = \frac{1}{[AL:L]} < 1$ if $L \subsetneq AL$. Similarly, we have $\lvert\det(\1A)\rvert > 1$ if $\1A\1L \subsetneq \1L$, and $\lvert\det(\1A)\rvert < 1$ if $\1L \subsetneq \1A\1L$. But we know that $\lvert\det(\1A)\rvert = \lvert\det(BA^\varepsilon B^{-1})\rvert = \lvert\det(A)\rvert^\varepsilon$; therefore, we have $\varepsilon = 1$ if and only if either $AL \subsetneq L$ and $\1A\1L \subsetneq \1L$, or $L \subsetneq AL$ and $\1L \subsetneq \1A\1L$.

Suppose first that $\varepsilon = 1$, and so $\1A = BAB^{-1}$. If $AL \subsetneq L$ and $\1A\1L \subsetneq \1L$, then we have $\1L = \1A\1L \cup \1L = B(AL \cup L) = BL$, as required. Otherwise, we have $L \subsetneq AL$ and $\1L \subsetneq \1A\1L$, implying that $BAB^{-1}\1L = \1A\1L = \1A\1L \cup \1L = B(AL \cup L) = BAL$, and so $\1L = (BAB^{-1})^{-1} BAL = BL$, as required.

Suppose now that $\varepsilon = -1$, and so $\1A = BA^{-1}B^{-1}$. If $L \subsetneq AL$ and $\1A\1L \subsetneq \1L$, then we have $\1L = \1A\1L \cup \1L = B(AL \cup L) = BAL$, as required. Otherwise, we have $AL \subsetneq L$ and $\1L \subsetneq \1A\1L$, implying that $BA^{-1}B^{-1}\1L = \1A\1L = \1A\1L \cup \1L = B(AL \cup L) = BL$, and so $\1L = (BA^{-1}B^{-1})^{-1} BL = BAL$, as required. \qedhere

\end{description}

\end{proof}

\section{Non-polycyclic metabelian groups} \label{sec:metab}

Here we consider isomorphisms between metabelian Leary--Minasyan groups that are not polycyclic. We first describe the the subgroup $K$ and the map $K \to K, g \mapsto tgt^{-1}$ appearing in Lemma~\ref{lem:ascHNN} in the case $H \cong \Z^n$. This will allow us to classify metabelian Leary--Minasyan groups up to isomorphism.

\begin{lem} \label{lem:metab-iso-new}
Let $n \geq 0$ and let $A \in GL_n(\Q)$ be a matrix with integer entries. Let $G = G(A,\Z^n)$ be an ascending HNN-extension of $H \cong \Z^n$, and let $K$ and $t$ be as in Lemma~\ref{lem:ascHNN}. Then there is an injective group homomorphism $\phi\colon K \to \Q^n$ such that $\phi(K) = \bigcup_{j \in \Z} A^j(\Z^n)$ and such that $\phi(tgt^{-1}) = A\phi(g)$ for all $g \in K$.
\end{lem}

\begin{proof}
By Lemma~\ref{lem:ascHNN}, $K$ is precisely the set of elements of $G$ of the form $t^j \mathbf{x}^{\mathbf{v}} t^{-j}$ for some $j \in \Z$ and $\mathbf{v} \in \Z^n$. We then define $\phi\colon K \to \Q^n$ by setting $\phi(t^j \mathbf{x}^{\mathbf{v}} t^{-j}) = A^j \mathbf{v}$. A direct computation shows that $\phi(g)$ is independent of the choice of a word $t^j \mathbf{x}^{\mathbf{v}} t^{-j}$ representing $g$, and that $\phi(gh) = \phi(g)+\phi(h)$ for all $g,h \in K$; thus $\phi$ is a well-defined group homomorphism. It is clear from the construction that $\phi(K) = \bigcup_{j \in \Z} A^j(\Z^n)$. To show that $\phi$ is injective, note that if $\phi(g) = \mathbf{0}$ for some $g = t^j \mathbf{x}^{\mathbf{v}} t^{-j} \in K$ then we have $\mathbf{v} = A^{-j} \mathbf{0} = \mathbf{0}$ and so $g = t^j \mathbf{x}^{\mathbf{0}} t^{-j} = t^j t^{-j} = 1$. Finally, given any $g = t^j \mathbf{x}^{\mathbf{v}} t^{-j} \in K$ we have
\[
\phi(tgt^{-1}) = \phi(t^{j+1} \mathbf{x}^{\mathbf{v}} t^{-(j+1)}) = A^{j+1} \mathbf{v} = A (A^j \mathbf{v}) = A\phi(g),
\]
as required.
\end{proof}

The following result proves the `only if' direction of Theorem~\ref{thm:main} in the case when $G(A,L)$ is metabelian and not polycyclic. By Lemma~\ref{lem:polycyclic}, given a matrix $A \in GL_n(\Q)$ with integer entries, the metabelian group $G(A,\Z^n)$ is polycyclic if and only if $A \in GL_n(\Z)$.

\begin{prop} \label{prop:main-metab}
Let $n,\1n \geq 0$, and let $A \in GL_n(\Q) \setminus GL_n(\Z)$ and $\1A \in GL_{\1n}(\Q) \setminus GL_{\1n}(\Z)$ be matrices with integer entries. Suppose that $G(A,\Z^n) \cong G(\1A,\Z^{\1n})$. Then $n = \1n$, and there exists $B \in GL_n(\Q)$ such that $\1A = BAB^{-1}$ and $\bigcup_{j \in \Z} A^j(\Z^n) = \bigcup_{j \in \Z} A^jB^{-1}(\Z^n)$.
\end{prop}

\begin{proof}
Let $G = G(A,L)$ be the HNN-extension of $H \cong \Z^n$ with stable letter $t$, and let $K = \bigcup_{j \in \Z} t^jHt^{-j} \lhd G$. Let $\phi\colon K \to \Q^n$ be the map given by Lemma~\ref{lem:metab-iso-new}. We set $\1G := G(\1A,\1L)$, and define the subgroups $\1H < \1K \lhd \1G$, the element $\1t \in \1G$ and the map $\1\phi\colon \1K \to \Q^{\1n}$ in an analogous way.

Let $\Phi\colon G \to \1G$ be an isomorphism. Since $A \notin GL_n(\Z)$, the group $G$ (and so $\1G$) is not polycyclic by Lemma~\ref{lem:polycyclic}, and so Proposition~\ref{prop:elliptic} implies that $\Phi$ sends elliptic elements of $G$ (with respect to the action on $T(A,L)$) to elliptic elements of $\1G$ (with respect to the action on $T(\1A,\1L)$). But as $K$ (respectively $\1K$) is a normal subgroup consisting of conjugates of elements in $H$ (respectively $\1H$), it follows that $K$ (respectively $\1K$) is precisely the set of elliptic elements with respect to $G \curvearrowright T(A,L)$ (respectively $\1G \curvearrowright T(\1A,\1L)$). Thus $\Phi(K) = \overline{K}$.

Now consider the map $\beta = \1\phi \circ \Phi \circ \phi^{-1}\colon \phi(K) \to \1\phi(\1K)$, which is a well-defined isomorphism since $\phi$ and $\1\phi$ are injective and since $\Phi|_K\colon K \to \1K$ is an isomorphism. Since $\Z^n \subseteq \phi(K) \subseteq \Q^n$ and $\Z^{\1n} \subseteq \1\phi(\1K) \subseteq \Q^{\1n}$, we have $\phi(K) \otimes \Q = \Q^n$ and $\1\phi(\1K) \otimes \Q = \Q^{\1n}$, implying that $\beta$ extends to an isomorphism $\beta \otimes \Q\colon \Q^n \to \Q^{\1n}$. In particular, $n = \1n$, and $\beta \otimes \Q$ is represented by a matrix $B \in GL_n(\Q)$ such that $\1\phi(\Phi(k)) = B\phi(k)$ for all $k \in K$.

Let $\1q\colon \1G \to \1G/\1K$ be the quotient map. Then the map $\1q \circ \Phi\colon G \to \1G/\1K = \langle \1K\1t \rangle \cong \Z$ is surjective since $\Phi$ and $\1q$ are surjective; therefore, since $G = \langle K,t \rangle$ and $\Phi(K) = \1K$, the cyclic group $\1G/\1K$ is generated by $\1q(\Phi(t))$. It follows that $\Phi(t) = \1g \1t^{\varepsilon}$ for some $\1g \in \1K$ and $\varepsilon \in \{ \pm 1 \}$. Now given $\1k \in \1K$, we may compute that
\begin{align*}
B^{-1}\1\phi(\1k) &= \phi(\Phi^{-1}(\1k)) = A^{-1} \phi(t \Phi^{-1}(\1k) t^{-1}) = A^{-1} \phi(\Phi^{-1}(\1g \cdot \1t^{\varepsilon} \1k \1t^{-\varepsilon} \cdot \1g^{-1})) \\ &= A^{-1} \phi(\Phi^{-1}(\1t^{\varepsilon} \1k \1t^{-\varepsilon})) = A^{-1}B^{-1} \1\phi(\1t^{\varepsilon} \1k \1t^{-\varepsilon}) = A^{-1}B^{-1}\1A^{\varepsilon} \1\phi(\1k).
\end{align*}
Since $\Z^n \subseteq \1\phi(\1K)$, this implies that $B^{-1}\mathbf{v} = A^{-1}B^{-1}\1A^{\varepsilon}\mathbf{v}$ for all $\mathbf{v} \in \Z^n$, and so $B^{-1} = A^{-1}B^{-1}\1A^\varepsilon$, i.~e.\ $\1A = B A^\varepsilon B^{-1}$. Furthermore, as $A$ has integer entries and $A \notin GL_n(\Z)$ we have $\lvert\det(A)\rvert \geq 2$, and similarly $\lvert\det(\1A)\rvert \geq 2$; as $\lvert\det(\1A)\rvert = \lvert\det(A^\varepsilon)\rvert = \lvert\det(A)\rvert^\varepsilon$, it follows that $\varepsilon = 1$ and so $\1A = BAB^{-1}$.

Finally, since $\1\phi(\Phi(k)) = B\phi(k)$ for all $k \in K$ and since $\Phi(K) = \1K$, we have $\1\phi(\1K) = B\phi(K)$. It follows that
\[
\bigcup_{j \in \Z} A^j(\Z^n) = \phi(K) = B^{-1}\1\phi(\1K) = \bigcup_{j \in \Z} B^{-1}\1A^j(\Z^n) = \bigcup_{j \in \Z} A^jB^{-1}(\Z^n),
\]
as required.
\end{proof}

\begin{rmk} \label{rmk:metab}
To show that case~\ref{it:main-metab} in Theorem~\ref{thm:main} is not covered by case~\ref{it:main-gen}, consider the following example. Let $A = \begin{pmatrix} 0 & 1 \\ 8 & 0 \end{pmatrix} \in GL_2(\Q)$ and $\1A = \begin{pmatrix} 0 & 2 \\ 4 & 0 \end{pmatrix} \in GL_2(\Q)$. Then we have $\1A = BAB^{-1}$ and $\bigcup_{j \in \Z} A^j(\Z^2) = \Z[1/2]^2 = \bigcup_{j \in \Z} A^jB^{-1}(\Z^2)$ for $B = \begin{pmatrix} 2 & 0 \\ 0 & 1 \end{pmatrix} \in GL_2(\Q)$, implying that $G(A,\Z^2) \cong G(\1A,\Z^2)$. However, a straightforward direct computation shows that $\1A \neq BA^\varepsilon B^{-1}$ for any $B \in GL_2(\Z)$ and $\varepsilon \in \{ \pm 1 \}$.

To show that the equality $\bigcup_{j \in \Z} A^j(\Z^n) = \bigcup_{j \in \Z} A^jB^{-1}(\Z^n)$ in case~\ref{it:main-metab} cannot be dropped, consider the following. Let $A = \begin{pmatrix} 1 & 0 \\ 2 & 1 \end{pmatrix} \in GL_2(\Q)$ and $\1A = \begin{pmatrix} 1 & 0 \\ 1 & 1 \end{pmatrix} \in GL_2(\Q)$. We then have $\1A = BAB^{-1}$ for $B = \begin{pmatrix} 2 & 0 \\ 0 & 1 \end{pmatrix} \in GL_2(\Q)$; however, the groups $G(A,\Z^2)$ and $G(\1A,\Z^2)$ are not isomorphic since they have non-isomorphic abelianisations.
\end{rmk}

\section{Polycyclic groups} \label{sec:polyc}

In this section, we assume that $L = AL = \Z^n$, and consequently $A \in GL_n(\Z)$. It then follows from Lemma~\ref{lem:ascHNN} that $G(A,\Z^n) = H \rtimes \langle t \rangle$, where $H = \langle x_1,\ldots,x_n \rangle \cong \Z^n$, with the map $H \to H, h \mapsto tht^{-1}$ represented by the matrix $A$.

\subsection{Isomorphisms}

Here we give a criterion for two semidirect products of $H \cong \Z^n$ and $\langle t \rangle \cong \Z$ to be isomorphic.

Given two elements $x,y \in G$ of a group $G$, we write $[x,y] := xyx^{-1}y^{-1}$. More generally, given two subsets $S,T \subseteq G$ we write $[S,T]$ for $\{ [s,t] \mid s \in S, t \in T \} \subseteq G$, and we write $[g,T]$ for $[\{g\},T]$ (where $g \in G$ and $T \subseteq G$). This notation is slightly non-standard: if $H,K \leq G$ are subgroups, we only write $[H,K]$ for the \emph{subset} $\{ [h,k] \mid h \in H, k \in K \}$, not for the subgroup generated by it.

\begin{lem} \label{lem:2ways}
Let $G$ be a group with subgroups $H,\1H < G$, both isomorphic to $\Z^n$, and infinite order elements $t,\1t \in G$ such that $G = H \rtimes \langle t \rangle = \1H \rtimes \langle \1t \rangle$. Let $N = H\1H$ and $E = H \cap \1H$, and suppose that $[G:N] < \infty$. Then $G/E \cong \Z^2$, and $H = E \times \langle h \rangle$ for some $h \in H$. Moreover, for any such $h \in H$, there exists $\1h \in \1H$ such that $\1H = E \times \langle \1h \rangle$ and $[t,h] = [\1t,\1h]$.
\end{lem}

\begin{proof}
Throughout this proof, we will make use of the commutator identities
\begin{equation} \label{eq:comid}
[x,yz] = [x,y] \cdot y [x,z] y^{-1} \qquad \text{and} \qquad [xy,z] = x [y,z] x^{-1} \cdot [x,z].
\end{equation}

Note first that $N/\1H$ has finite index in $G/\1H \cong \Z$, and so $H/E \cong N/\1H \cong \Z$. This implies that $E$ is a direct factor of the free abelian group $H$, and so $H = E \times \langle h \rangle$ for some $h \in H$. Note also that as both $G/H \cong \Z$ and $G/\1H \cong \Z$ are abelian, it follows that $G' \subseteq H \cap \1H = E$ and so $G/E$ is abelian; together with the fact that $G/H \cong H/E \cong \Z$, this implies that $G/E \cong \Z^2$.

In order to show the last statement, we first claim that the subgroup $\widehat{E} \leq G$ generated by $[G,E]$ is equal to $[\1t,E]$. It is clear that $[\1t,e] \in \widehat{E}$ for all $e \in E$, so we only need to show that $\widehat{E} \subseteq [\1t,E]$. But since $E$ is abelian and normal, the first equation in \eqref{eq:comid} implies that $[\1t,ef] = [\1t,e] [\1t,f]$ for all $e,f \in E$, and so the map $E \to E, e \mapsto [\1t,e]$ is a group homomorphism; in particular, $[\1t,E]$ is a subgroup of $G$. It is therefore enough to express $[g,e]$ as a product of elements of the form $[\1t,f]$ or $[\1t,f]^{-1}$, where $f \in E$, for any given $g \in G$ and $e \in E$.

Now since $G/\1H = \langle \1t\1H \rangle$, we have $g = \1t^p \1k$ for some $p \in \Z$ and $\1k \in \1H$. As $E \subseteq \1H$ and $\1H$ is abelian, we have $[g,e] = [\1t^p,e]$. If $p \geq 0$, then the second equation in \eqref{eq:comid} implies that, by induction on $p$,
\begin{equation} \label{eq:tpe}
[\1t^p,e] = \prod_{i=1}^p \1t^{p-i} [\1t,e] \1t^{i-p} = \prod_{i=1}^p [\1t,\1t^{p-i} e \1t^{i-p}],
\end{equation}
which has the required form. If $p < 0$, then we have
\[
[\1t^p,e] = \1t^p [e,\1t^{-p}] \1t^{-p} = [ \1t^{p} e \1t^{-p}, \1t^{-p} ] = [ \1t^{-p}, \1t^{p} e \1t^{-p} ]^{-1},
\]
and so we are again done by \eqref{eq:tpe}. This proves that $\widehat{E} = [\1t,E]$, as claimed.

Now as $E$ is a direct factor of $H$, it is also a direct factor of $\1H$ by an analogous argument. Thus there exists $\1h \in \1H$ such that $\1H = E \times \langle \1h \rangle$. We now aim to modify $\1h$ so that we have $[t,h] = [\1t,\1h]$.

Note that we can write $\1t = t^q h^p e$ and $\1h = t^s h^r f$ for some $p,q,r,s \in \Z$ and $e,f \in E$. As $G/H = \langle tH \rangle$ and $H/E = \langle h E \rangle$, it follows that the pair $\{ h E, t E \}$ generates the group $G/E \cong \Z^2$; similarly, the pair $\{ \1h E, \1t E \}$ also generates $G/E$. This implies that $qr-ps = \det\begin{pmatrix} q&p\\s&r \end{pmatrix} \in \{ \pm 1 \}$. Note that $\1h^{-1} = t^{-s} h^{-r} f_0$ for some $f_0 \in E$; therefore, after replacing $\1h$ with $\1h^{-1}$ if necessary, we can assume that $qr-ps = 1$.

Now since $G/E$ is abelian, the group $G/\widehat{E}$ is $2$-step nilpotent, i.~e.\ every commutator in $G/\widehat{E}$ is central. Note also that $e\widehat{E}$ and $f\widehat{E}$ are central in $G/\widehat{E}$. In this case, \eqref{eq:comid} implies that
\[
[\1t,\1h]\widehat{E} = [t^qh^pe,t^sh^rf]\widehat{E} = [t,t]^{qs} [t,h]^{qr} [h,t]^{ps} [h,h]^{pr} \widehat{E} = [t,h]^{qr-ps} \widehat{E} = [t,h] \widehat{E},
\]
and so, as $\widehat{E} = [\1t,E]$, we have $[t,h] = [\1t,\1h] [\1t,g]$ for some $g \in E$. But as $[\1t,g],\1h \in \1H$ and as $\1H$ is abelian, the first equation in \eqref{eq:comid} then implies that $[t,h] = [\1t,\1h g]$. Since $\1H = E \times \langle \1h g \rangle$, we may replace $\1h$ with $\1h g$ to get $[t,h] = [\1t,\1h]$, as required.
\end{proof}

\begin{lem} \label{lem:iso=>iii}
Let $A,\1A \in GL_n(\Z)$, and suppose that $G(A,\Z^n) \cong G(\1A,\Z^n)$ and that $\1A$ is not conjugate to $A^{\pm 1}$ in $GL_n(\Z)$. Then there exists a matrix $C \in GL_{n-1}(\Z)$ of order $m_0 < \infty$ such that $A$ and $\1A$ are conjugate in $GL_n(\Z)$ to $\begin{pmatrix} 1 & \mathbf{0}^T \\ \mathbf{u} & C \end{pmatrix}$ and $\begin{pmatrix} 1 & \mathbf{0}^T \\ \mathbf{u} & C^q \end{pmatrix}$, respectively, for some $\mathbf{u} \in \Z^{n-1}$ and some $q \in \Z$ with $\gcd(q,m_0) = 1$.
\end{lem}

\begin{proof}
Let $G = G(A,\Z^n) \cong G(\1A,\Z^n)$. Then $G = H \rtimes \langle t \rangle = \1H \rtimes \langle \1t \rangle$, where $H \cong \Z^n \cong \1H$ and $t,\1t \in G$ have infinite order, and where the maps $H \to H, h \mapsto tht^{-1}$ and $\1H \to \1H, h \mapsto \1th\1t^{-1}$ are represented by the matrices $A$ and $\1A$, respectively. Let $N = H\1H$, and note that $N/H \leq G/H \cong \Z$ and $N/\1H \leq G/\1H \cong \Z$.

Suppose first that $[G:N] = \infty$. Then $[G/H : N/H] = \infty$ and so $N/H$ is trivial, i.~e.\ $N = H$; similarly, $N = \1H$. But then $G/N \cong \Z$ is generated by each $tN$ and $\1tN$, implying that $tN = (\1tN)^\varepsilon$ for some $\varepsilon \in \{\pm 1 \}$, and so $\1t = t^\varepsilon h$ for some $h \in N$. Thus, as $N$ is abelian, the map $N \to N, g \mapsto \1tg\1t^{-1}$ is represented by $A^\varepsilon$, contradicting the fact that $\1A$ is not conjugate to $A^{\pm 1}$ in $GL_n(\Z)$.

Therefore, we must have $[G:N] = m < \infty$. Let $E = H \cap \1H$. It then follows from Lemma~\ref{lem:2ways} that $G/E$ is abelian, and that there exist $h \in H$ and $\1h \in \1H$ such that $H = E \times \langle h \rangle$, $\1H = E \times \langle \1h \rangle$, and $[t,h] = [\1t,\1h]$.

Let $\{ e_1,\ldots,e_{n-1} \}$ be a basis for $E$, so that $\{ h,e_1,\ldots,e_{n-1} \}$ and $\{ \1h,e_1,\ldots,e_{n-1} \}$ are bases for the free abelian groups $H$ and $\1H$, respectively. The fact that $G/E$ is abelian, and in particular that $[t,H] \subseteq E$ and $[\1t,\1H] \subseteq E$, implies that the maps $H \to H, k \mapsto tkt^{-1}$ and $\1H \to \1H, k \mapsto \1tk\1t^{-1}$ are represented -- with respect to these bases -- by matrices $\begin{pmatrix} 1 & \mathbf{0}^T \\ \mathbf{u} & C \end{pmatrix}$ and $\begin{pmatrix} 1 & \mathbf{0}^T \\ \mathbf{v} & D \end{pmatrix}$, respectively, for some $\mathbf{u},\mathbf{v} \in \Z^{n-1}$ and $C,D \in GL_{n-1}(\Z)$. The fact that $[t,h] = [\1t,\1h]$ implies that $\mathbf{u} = \mathbf{v}$, whereas since $\1t = t^q k$ for some $q \in \Z$ and $k \in H$, the fact that $H$ is abelian implies that $\1te\1t^{-1} = t^qet^{-q}$ for all $e \in E$, and so $D = C^q$.

It is therefore enough to show that $C$ has finite order $m_0$ coprime to $q$. But note that the subgroup $E = H \cap \1H$ is central in $N = H\1H$ since both $H$ and $\1H$ are abelian; as $[G:N] = m$, we have $t^m \in N$ and so $[t^m,e] = 1$ for all $e \in E$, implying that $C^m = I_{n-1}$, and so the order $m_0$ of $C$ is finite and divides $m$. Finally, the finite cyclic group $G/N \cong \Z/m\Z$ is a quotient of each $G/H = \langle tH \rangle$ and $G/\1H = \langle \1t\1H \rangle$, and so it is generated by each $tN$ and $\1t N = t^qN$. It follows that $\gcd(q,m) = 1$, and so $\gcd(q,m_0) = 1$, as required.
\end{proof}

\begin{lem} \label{lem:iii=>iso}
Let $q,m,n \in \Z$ be such that $m,n > 0$ and $\gcd(q,m) = 1$, let $\mathbf{u} \in \Z^{n-1}$, and let $C \in GL_{n-1}(\Z)$ be a matrix of order $m$. Write $A = \begin{pmatrix} 1 & \mathbf{0}^T \\ \mathbf{u} & C \end{pmatrix}$ and $\1A = \begin{pmatrix} 1 & \mathbf{0}^T \\ \mathbf{u} & C^q \end{pmatrix}$. Then $G(A,\Z^n) \cong G(\1A,\Z^n)$.
\end{lem}

\begin{proof}
Since $\gcd(q,m) = 1$, there exist $p,r \in \Z$ such that $mp+qr = 1$. Let $G = G(A,\Z^n) = H \rtimes \langle t \rangle$, where $H \cong \Z^n$ and $\langle t \rangle \cong \Z$, so that the map $H \to H, k \mapsto tkt^{-1}$ is represented by the matrix $A$ with respect to a basis $\{ h,e_1,\ldots,e_{n-1} \}$ of $H$. We then have $[t,H] \subseteq E$, where $E = \langle e_1,\ldots,e_{n-1} \rangle$, implying that the map $\varphi$, defined by setting $\varphi(t) = r$, $\varphi(h) = m$ and $\varphi(e_i) = 0$ for each $i$, extends to a homomorphism $\varphi\colon G \to \Z$. Moreover, $\varphi$ is surjective since $\varphi(\1t) = 1$, where $\1t = t^q h^p$, and the kernel of $\varphi$, which is easily seen to be equal to $\1H := \langle t^{-m}h^r,e_1,\ldots,e_{n-1} \rangle$, is abelian since $C^{-m} = I_{n-1}$, and consequently $[t^{-m}h^r,e_i] = [t^{-m},e_i] = 1$ for all $i$. It is clear that $\1H/E = \langle t^{-m}h^rE \rangle \cong \Z$ and consequently $\1H \cong \Z^n$, and that $G = \1H \rtimes \langle \1t \rangle$. Moreover, the group $N := H\1H = \langle h,e_1,\ldots,e_{n-1},t^m \rangle$ has index $m < \infty$ in $G$.

It then follows from Lemma~\ref{lem:2ways} that there exists $\1h \in \1H$ such that $[t,h] = [\1t,\1h]$ and such that $\{ \1h,e_1,\ldots,e_{n-1} \}$ is a basis of the free abelian group $\1H$. It is easy to check that, with respect to this basis, the map $\1H \to \1H, k \mapsto \1tk\1t^{-1}$ is represented by the matrix $\1A$. Thus $G(A,\Z^n) \cong G(\1A,\Z^n)$, as required.
\end{proof}

\begin{rmk} \label{rmk:2step}
In the proofs of Lemmas \ref{lem:iso=>iii} and \ref{lem:iii=>iso}, we have constructed subgroups $E \lhd N$ of $G = G(A,\Z^n)$ such that $E$ is central in $N$, such that $G/E$ (and so $N/E$) is abelian, and such that $N$ has finite index in $G$. It follows that any group $G$ satisfying the assumptions of these Lemmas (and so the condition~\ref{it:main-polyc} in Theorem~\ref{thm:main}) must be virtually $2$-step nilpotent.

However, such a group $G$ need not be virtually abelian: indeed, for $A = \begin{pmatrix} 1&0\\1&1 \end{pmatrix}$, the group $G(A,\Z^2)$ is the integral Heisenberg group, which is not virtually abelian.
\end{rmk}

\subsection{Conjugacy over \texorpdfstring{$\Q$}{Q}}

Here we discuss the implications of condition~\ref{it:main-polyc} in Theorem~\ref{thm:main}. In particular, we prove that any matrices $A$ and $\1A$ satisfying this condition are conjugate in $GL_n(\Q)$: see Lemma~\ref{lem:conj/Q}. As shown in Example~\ref{ex:polyc}, such an $\1A$ need not be conjugate to $A^{\pm 1}$ in $GL_n(\Z)$.

\begin{lem} \label{lem:fo-conj}
Let $D \in GL_n(\Q)$ be a matrix of order $m < \infty$, and let $q \in \Z$ be such that $\gcd(q,m) = 1$. Then $D$ and $D^q$ are conjugate in $GL_n(\Q)$.
\end{lem}

\begin{proof}
Note first that $D$ and $D^q$ have the same invariant subspaces of $\Q^n$. Indeed, clearly if a subspace of $\Q^n$ is $D$-invariant then it is also $D^q$-invariant; the converse holds since $D$ and $D^q$ generate the same cyclic subgroup in $GL_n(\Q)$ (as $\gcd(q,m) = 1$), implying that $D = (D^q)^r$ for some $r \in \Z$. Now as $\Q$ has characteristic zero, any reducible representation of $\Z/m\Z$ over $\Q$ is decomposable, implying that we can write $\Q^n = P_1 \oplus \cdots \oplus P_\ell$, where each $P_i$ is $D$-invariant (and so $D^q$-invariant) and has no proper non-zero $D$-invariant (and so $D^q$-invariant) subspace. It is therefore enough to show that $D|_{P_i}$ and $D^q|_{P_i}$ are conjugate in $GL(P_i)$: that is, it is enough to consider the case when $D$ (and so $D^q$) is \emph{irreducible}, i.~e.\ when $\Q^n$ has no proper non-zero $D$-invariant subspace.

Now two irreducible matrices are conjugate in $GL_n(\Q)$ if and only if they have the same characteristic polynomial: see \cite[Theorem~12.17]{dummit-foote}. Thus, let $\lambda_1,\ldots,\lambda_n \in \C$ be the eigenvalues of $D$ (counted with multiplicities), so that $\chi_D(X) = \prod_{i=1}^n (X-\lambda_i)$. Since $D^m = I_n$, each $\lambda_i$ is an $m$-th root of unity. Let $\zeta_m$ be a primitive $m$-th root of unity, and let $\sigma \in \Gal(\Q(\zeta_m)/\Q)$ be defined by $\sigma(\zeta_m) = \zeta_m^q$. If we write $\chi_D(X) = X^n + a_{n-1}X^{n-1} + \cdots + a_0$, we then have
\begin{align*}
\chi_{D^q}(X) &= \prod_{i=1}^n (X-\lambda_i^q) = \prod_{i=1}^n (X-\sigma(\lambda_i)) = X^n + \sigma(a_{n-1})X^{n-1} + \cdots + \sigma(a_0) \\ &= X^n + a_{n-1}X^{n-1} + \cdots + a_0 = \chi_D(X)
\end{align*}
since $a_0,\ldots,a_{n-1} \in \Q$. Thus $\chi_D(X) = \chi_{D^q}(X)$, as required.
\end{proof}

\begin{lem} \label{lem:conj/Q}
Let $q,m,n \in \Z$ be such that $m,n > 0$ and $\gcd(q,m) = 1$, let $\mathbf{u} \in \Z^{n-1}$, and let $C \in GL_{n-1}(\Z)$ be a matrix of order $m$. Then the matrices $A = \begin{pmatrix} 1 & \mathbf{0}^T \\ \mathbf{u} & C \end{pmatrix}$ and $\1A = \begin{pmatrix} 1 & \mathbf{0}^T \\ \mathbf{u} & C^q \end{pmatrix}$ are conjugate in $GL_n(\Q)$.
\end{lem}

\begin{proof}
We have a direct sum decomposition $\Q^n = \langle \mathbf{w} \rangle \oplus P$ such that $A\mathbf{w}-\mathbf{w} = \1A\mathbf{w}-\mathbf{w} = \mathbf{u} \in P$ and such that $P \leq \Q^n$ is an $A$-invariant and $\1A$-invariant subspace with $A|_P$ and $\1A|_P$ represented by $C$ and $C^q$, respectively. Since $\Q$ is a field of characteristic zero and $C \in GL_{n-1}(\Z)$ has finite order, we have a decomposition $P = Q \oplus R$, where $Q$ and $R$ are $C$-invariant and $Q$ is the $1$-eigenspace of $C$ (in particular, $C|_R$ has no eigenvalues equal to one). As $C$ and $C^q$ generate the same cyclic subgroup of $GL_{n-1}(\Q)$, it follows that $Q$ is also the $1$-eigenspace of $C^q$.

Now we can write $\mathbf{u} = \mathbf{u}_Q + \mathbf{u}_R$ for some $\mathbf{u}_Q \in Q$ and $\mathbf{u}_R \in R$. Since $C|_R$ has no eigenvalues equal to $1$, it follows that the matrix $I_{\dim(R)}-C|_R$ is non-singular, and so $\mathbf{u}_R = \mathbf{v}_R - C\mathbf{v}_R$ for some $\mathbf{v}_R \in R$. But then we have
\[
A(\mathbf{w}+\mathbf{v}_R) - (\mathbf{w}+\mathbf{v}_R) = (A\mathbf{w}-\mathbf{w}) + (C\mathbf{v}_R-\mathbf{v}_R) = \mathbf{u} - \mathbf{u}_R = \mathbf{u}_Q.
\]
Similarly, there exists $\1{\mathbf{v}}_R \in R$ such that $\1A(\mathbf{w}+\1{\mathbf{v}}_R) - (\mathbf{w}+\1{\mathbf{v}}_R) = \mathbf{u}_Q$.

Consider the direct sum decompositions $\Q^n = S \oplus R = \1S \oplus R$, where $S = \langle \mathbf{w} + \mathbf{v}_R \rangle \oplus Q$ and $\1S = \langle \mathbf{w} + \1{\mathbf{v}}_R \rangle \oplus Q$. By construction, $S$ and $R$ are $A$-invariant, and the $A$-actions on $S$ and $R$ are represented by matrices $\begin{pmatrix} 1 & \mathbf{0}^T \\ \mathbf{u}_Q & I_N \end{pmatrix}$ and $C|_R$, respectively, where $N = \dim(Q)$. Similarly, $\1S$ and $R$ are $\1A$-invariant, and the $\1A$-actions on $\1S$ and $R$ are represented by matrices $\begin{pmatrix} 1 & \mathbf{0}^T \\ \mathbf{u}_Q & I_N \end{pmatrix}$ and $C^q|_R$, respectively. In order to show that $A$ and $\1A$ are conjugate in $GL_n(\Q)$, it is therefore enough to show that $C|_R$ and $C^q|_R$ are conjugate in $GL(R)$. As by construction $C|_R$ and $C$ have the same order, and so the order of $C|_R$ is finite and coprime to $q$, the conjugacy of $C|_R$ and $C^q|_R = (C|_R)^q$ in $GL(R)$ follows from Lemma~\ref{lem:fo-conj}.
\end{proof}

Finally, we show that the matrices $A$ and $\1A$ in Lemma~\ref{lem:conj/Q} need not be conjugate in $GL_n(\Z)$.

\begin{ex} \label{ex:polyc}
There exists a matrix $C \in GL_n(\Z)$ of order $m < \infty$ such that, for some $q \in \Z$ with $\gcd(q,m) = 1$, neither $\begin{pmatrix} 1 & \mathbf{0}^T \\ \mathbf{0} & C \end{pmatrix}$ nor $\begin{pmatrix} 1 & \mathbf{0}^T \\ \mathbf{0} & C \end{pmatrix}^{-1}$ is conjugate to $\begin{pmatrix} 1 & \mathbf{0}^T \\ \mathbf{0} & C^q \end{pmatrix}$ in $GL_{n+1}(\Z)$. More specifically, we could take here a matrix $C \in GL_{36}(\Z)$ of order $m = 37$, and $q = 6$.
\end{ex}

\begin{proof}
We first note that if $\det(I-C) \neq 0$, then it is enough to show that $C$ and $C^q$ are not conjugate in $GL_n(\Z)$. Indeed, if $C$ was such a matrix with $\det(I-C) \neq 0$, then the $1$-eigenspace (in $\Q^{n+1}$) of each $A = \begin{pmatrix} 1 & \mathbf{0}^T \\ \mathbf{0} & C \end{pmatrix}$ and $\1A = \begin{pmatrix} 1 & \mathbf{0}^T \\ \mathbf{0} & C^q \end{pmatrix}$ would be $1$-dimensional spanned by the vector $(1,0,\ldots,0)$. Therefore, if we had $BA^{\pm 1}B^{-1} = \1A$ for some $B \in GL_{n+1}(\Z)$ then $B$ would have to fix this $1$-dimensional subspace, and so $B = \begin{pmatrix} u & \mathbf{v}^T \\ \mathbf{0} & W \end{pmatrix}$ for some $u \in \{\pm 1\}$, $\mathbf{v} \in \Z^n$ and $W \in GL_n(\Z)$. We could then compute that $\1A = BA^{\pm 1}B^{-1} = \begin{pmatrix} 1 & \mathbf{v}^TC^{\pm 1}W^{-1}-\mathbf{v}^TW^{-1} \\ \mathbf{0} & WC^{\pm 1}W^{-1} \end{pmatrix}$, and so $C^q = WC^{\pm 1}W^{-1}$. Thus, showing that $C^{\pm 1}$ and $C^q$ are not conjugate is enough.

Now given any irreducible polynomial $f(X) \in \Z[X]$ of degree $n$, a theorem of Latimer and MacDuffee \cite{latimer-macduffee} says that there exists a bijection between the $GL_n(\Z)$-conjugacy classes of matrices in $M_n(\Z)$ with characteristic polynomial $f(X)$ and the ideal class group of the field $\Q[X]/(f(X))$, where $M_n(\Z)$ is the set of $n \times n$ matrices with integer entries. Here, the \emph{ideal class group} $C(K)$ of a number field $K$ is the quotient of the group (under multiplication) of non-zero fractional ideals of $K$ by the subgroup of principal fractional ideals, where a \emph{fractional ideal} of $K$ is just a subset of the form $\alpha\mathfrak{a}$ for $\alpha \in K^\times$ and $\mathfrak{a} \unlhd \mathcal{O}_K$ an ideal in the ring of integers of $K$. The bijection is given by associating the class of a fractional ideal $\mathfrak{b}$ to the conjugacy class of the matrix representing the map $\mathfrak{b} \to \mathfrak{b}, r \mapsto rX$ under some choice of basis of $\mathfrak{b} \cong \Z^n$ as a free abelian group.

Let $K = \Q(\zeta_m) = \Q[X]/(f(X))$, where $\zeta_m$ is a primitive $m$-th root of unity for some odd prime $m \in \Z$, and $f(X)$ is the $m$-th cyclotomic polynomial. Fix an embedding $K \subset \C$; note that, since $K/\Q$ is Galois, $K$ is fixed under complex conjugation. We write $K_+$ for $K \cap \R = \Q(\zeta_m + \zeta_m^{-1})$, and write $h_-(K) := \frac{|C(K)|}{|C(K_+)|}$ for the \emph{relative class number} of $K$. It can be shown that the map $\mathfrak{b} \mapsto \mathfrak{b}\mathcal{O}_K$ (where $\mathfrak{b}$ is a fractional ideal in $K_+$) induces an injective homomorphism of abelian groups $i\colon C(K_+) \to C(K)$ \cite[Theorem 4.14]{washington}, and in particular $h_-(K) = |C(K)/i(C(K_+))| \in \Z$. Moreover, it is easy to verify that the action of $\Gal(K/\Q)$ on $K$ makes $C(K)$ into a $\Z[\Gal(K/\Q)]$-module with $i(C(K_+))$ being a submodule, and that the complex conjugation $\iota \in \Gal(K/\Q)$ induces the automorphism of $C(K)/i(C(K_+))$ given by $c \mapsto -c$ (see \cite{schoof}). Therefore, the induced map $\iota_\ast\colon C(K) \to C(K)$ is non-identity whenever $c \neq -c$ for some $c \in C(K)/i(C(K_+))$, and so $\iota_\ast$ is non-identity whenever $h_-(K)$ is not a power of $2$.

Now let $m \in \Z$ be an odd prime such that $h_-(\Q(\zeta_m))$ is not a power of $2$ and such that $-1$ is a square in $\Z/m\Z$. Such primes exist, the smallest one being $m = 37$: we have $h_-(\Q(\zeta_{37})) = 37$ \cite[Tables, \S3]{washington} and $6^2 \equiv -1 \pmod{37}$. Let $K = \Q(\zeta_m)$, let $q \in \Z$ be such that $q^2 \equiv -1 \pmod{m}$, and define $\sigma \in \Gal(K/\Q)$ by setting $\sigma(\zeta_m) = \zeta_m^q$; note that $\sigma^2$ is the complex conjugation $\iota\colon K \to K$. Since $h_-(K)$ is not a power of $2$, there exists a fractional ideal $\mathfrak{c} \subset K$ such that $[\mathfrak{c}] \neq \iota_*([\mathfrak{c}]) = \sigma_*^2([\mathfrak{c}])$; this implies that $\sigma_*([\mathfrak{c}]) \notin \{ [\mathfrak{c}], \iota_*([\mathfrak{c}]) \}$.

Let $C \in M_{m-1}(\Z)$ be the matrix whose $GL_{m-1}(\Z)$-conjugacy class corresponds to the class $\iota_*([\mathfrak{c}])$ under the Latimer--MacDuffee bijection described above. Then $C^m = I_{m-1}$ since $\zeta_m^m = 1$ in $K$, and in particular $C \in GL_{m-1}(\Z)$. Moreover, since the characteristic polynomial of $C$ is the $m$-th cyclotomic polynomial $f(X) = \frac{1-X^m}{1-X}$, which is not a multiple of $1-X$, it follows that $\det(I-C) \neq 0$. Finally, $C^q$ represents the map $\sigma^2(\mathfrak{c}) \to \sigma^2(\mathfrak{c}), r \mapsto r\zeta_m^q = r\sigma(\zeta_m)$, and so it also represents the map $\sigma(\mathfrak{c}) \to \sigma(\mathfrak{c}), r' \mapsto r'\zeta_m$ (take $r = \sigma(r')$ to see this); similarly, $C^{-1} = C^{q^2}$ represents the map $\mathfrak{c} \to \mathfrak{c}, r'' \mapsto r''\zeta_m$. Since $[\sigma(\mathfrak{c})] \notin \{ [\mathfrak{c}], [\iota(\mathfrak{c})] \}$, it follows that $C^{\pm 1}$ is not conjugate to $C^q$ in $GL_{m-1}(\Z)$, as required.
\end{proof}

\section{The proof of Theorem~\ref{thm:main}} \label{sec:pfmain}

In this section, we combine the results from previous sections to prove Theorem~\ref{thm:main}.

\begin{proof}[Proof of Theorem~\ref{thm:main}]
Suppose first that $n = \1n$ and that one of the conditions~\ref{it:main-gen}--\ref{it:main-polyc} holds.

If \ref{it:main-gen} is true, then either $\1A = BAB^{-1}$ and $\1L = BL$, or $\1A = BA^{-1}B^{-1}$ and $\1L = BAL$. It is then easy to see from the presentation \eqref{eq:pres} that we have an isomorphism $\Phi\colon G(A,L) \to G(\1A,\1L)$, defined by $\Phi(\mathbf{x}^{\mathbf{v}}) = \mathbf{x}^{B\mathbf{v}}$ for $\mathbf{v} \in \Z^n$ and $\Phi(t) = t$ in the former case, and $\Phi(\mathbf{x}^{\mathbf{v}}) = \mathbf{x}^{B\mathbf{v}}$ for $\mathbf{v} \in \Z^n$ and $\Phi(t) = t^{-1}$ in the latter case.

If \ref{it:main-metab} is true, then we may use the isomorphisms $G(A,L) \cong G(A^{-1},AL)$ and $G(\1A,\1L) \cong G(\1A^{-1},\1A\1L)$, described in the previous paragraph, to assume (without loss of generality) that $L = \1L = \Z^n$. If $AL = \Z^n$ as well, then without loss of generality we may assume that $\varepsilon = 1$; otherwise, we have $\lvert \det(A) \rvert > 1$ and $\lvert \det(A) \rvert^\varepsilon = \lvert \det(\1A) \rvert \geq 1$, implying again that $\varepsilon = 1$. By Lemmas \ref{lem:ascHNN} and \ref{lem:metab-iso-new}, we then have $G(A,L) = K \rtimes \langle t \rangle$, and there exists an injective homomorphism $\phi\colon K \to \Q^n$ such that $\phi(K) = \bigcup_{j \in \Z} A^j(\Z^n)$ and such that $\phi(t g t^{-1}) = A \phi(g)$ for all $g \in K$. Now set $\1\phi = B \circ \phi\colon K \to \Q^n$. We may then verify that $B^{-1}\1\phi(tgt^{-1}) = AB^{-1}\1\phi(g)$, and so $\1\phi(tgt^{-1}) = \1A\1\phi(g)$, for all $g \in K$. Moreover, we have
\[
B^{-1}\1\phi(K) = \phi(K) = \bigcup_{j \in \Z} A^j(\Z^n) = \bigcup_{j \in \Z} A^jB^{-1}(\Z^n) = B^{-1} \bigcup_{j \in \Z} BA^jB^{-1}(\Z^n),
\]
and so $\1\phi(K) = \bigcup_{j \in \Z} \1A^j(\Z^n)$. Since a semidirect product $K \rtimes \langle t \rangle$ (where $t$ has infinite order) is determined uniquely up to isomorphism by the group $K$ and the map $K \to K, g \mapsto tgt^{-1}$, it follows from Lemma~\ref{lem:metab-iso-new} that $G(A,L) \cong G(\1A,\1L)$.

Finally, if \ref{it:main-polyc} is true, then we may use the isomorphisms $G(A,\Z^n) \cong G(BAB^{-1},\Z^n)$ for $B \in GL_n(\Z)$, described above, to assume (without loss of generality) that $A = \begin{pmatrix} 1 & \mathbf{0}^T \\ \mathbf{u} & C \end{pmatrix}$ and $\1A = \begin{pmatrix} 1 & \mathbf{0}^T \\ \mathbf{u} & C^q \end{pmatrix}$. Then $G(A,L) \cong G(\1A,\1L)$ by Lemma~\ref{lem:iii=>iso}.

Conversely, suppose that $G := G(A,L) \cong G(\1A,\1L)$. If $G$ is not metabelian, then, by Lemma~\ref{lem:polycyclic}, we have $L \neq \Z^n \neq AL$ and $\1L \neq \Z^{\1n} \neq \1A\1L$. It then follows from Proposition~\ref{prop:main-large} that $n = \1n$ and the condition~\ref{it:main-gen} is true. If $G$ is metabelian but not polycyclic, then, by Lemma~\ref{lem:polycyclic}, exactly one of $L$ and $AL$ is equal to $\Z^n$, and exactly one of $\1L$ and $\1A\1L$ is equal to $\Z^{\1n}$. It then follows from Proposition~\ref{prop:main-metab}, along with the isomorphisms $G(A,L) \cong G(A^{-1},AL)$ and $G(\1A,\1L) \cong G(\1A^{-1},\1A\1L)$, that $n = \1n$ and that the condition~\ref{it:main-metab} holds. Finally, if $G$ is polycyclic, then (again by Lemma~\ref{lem:polycyclic}) we have $L = AL = \Z^n$ and $\1L = \1A\1L = \Z^{\1n}$; moreover, $n+1 = \1n+1$ is the Hirsch length of $G$, and so $n = \1n$. It then follows from Lemma~\ref{lem:iso=>iii} that either condition~\ref{it:main-gen} or condition~\ref{it:main-polyc} is true.
\end{proof}

\bibliographystyle{amsalpha}
\bibliography{../../all}

\end{document}